\documentclass[preprint,11pt]{article}
\usepackage{fullpage}
\ifdefined\usebigfont

\usepackage{times}
\usepackage[fontsize=13pt]{scrextend}
\usepackage[left=1.56in,right=1.56in,top=1.74in,bottom=1.74in]{geometry}
\usepackage[applemac]{inputenc}
\else
\fi
\usepackage{amssymb,amsfonts,amsmath,amsthm,amscd,dsfont,mathrsfs}
\usepackage{graphicx,float,psfrag,epsfig,amssymb}
\usepackage[usenames,dvipsnames,svgnames,table]{xcolor}
\definecolor{darkgreen}{rgb}{0.0,0,0.9}
\usepackage[pagebackref,letterpaper=true,colorlinks=true,pdfpagemode=none,citecolor=OliveGreen,linkcolor=BrickRed,urlcolor=BrickRed,pdfstartview=FitH]{hyperref}
\usepackage{wrapfig}
\usepackage{relsize}
\usepackage{color}
\usepackage{pict2e}
\usepackage[tight]{subfigure}
\usepackage{algorithm}
\usepackage[noend]{algorithmic}
\usepackage{caption}
\usepackage{nameref}
\usepackage{makecell}
\usepackage[font={small}]{caption} 
\usepackage{float}
\usepackage{enumitem}

\DeclareMathAlphabet{\mathpzc}{OT1}{pzc}{m}{it}

\newtheorem{propo}{Proposition}[section]
\newtheorem{lemma}[propo]{Lemma}
\newtheorem{definition}[propo]{Definition}
\newtheorem{coro}[propo]{Corollary}
\newtheorem{thm}[propo]{Theorem}

\newtheorem{assumption}[propo]{Assumption}

\theoremstyle{definition}
\newtheorem{remark}[propo]{Remark}

\def\cF{{\cal F}}

\def\cH{{\cal H}}
\def\cC{{\cal C}}
\def\cG{{\cal G}}
\def\cE{{\cal E}}

\def\cS{{\cal S}}

\def\event{\mathcal{E}}

\def\cG{\mathcal{G}}

\def\eps{{\varepsilon}}
\def\prob{{\mathbb P}}
\def\E{{\mathbb E}}
\def\Var{{\rm Var}}

\def\cJ{\mathcal{J}}

\def\L0{{L_i}}

\def\de{{\rm d}}
\def\<{\langle}
\def\>{\rangle}

\def\hth{\widehat{\theta}}

\def\dth{\widehat{\theta}^{{\rm d}}}

\def\hSigma{\widehat{\Sigma}}

\def\hsigma{\widehat{\sigma}}

\def\supp{{\rm supp}}
\def\limsup{\rm{lim \, sup}}
\def\liminf{\rm{lim \, inf}}
\def\F{{\sf F}}
\def\ind{{\mathbb I}}

\def\F{{\sf F}}
\def\normal{{\sf N}}
\def\Lth{\widehat{\theta}^{n}}

\def\sT{{\sf T}}

\def\id{{\rm I}}

\def\proj{\mathcal{P}}

\def\event{\mathcal{E}}

\def\v*{v_i}
\def\T*{T_i}

\def\u*{u_i}
\def\F*{F_i}

\definecolor{olivegreen}{rgb}{0,0.6,0.4}

\def\cH{{\mathcal{H}}}

\def\htheta{\widehat{\theta}}

\newcommand{\BALD}{\begin{aligned}}
\newcommand{\EALD}{\end{aligned}}
\newcommand{\BALDS}{\begin{aligned*}}
\newcommand{\EALDS}{\end{aligned*}}
\newcommand{\BCAS}{\begin{cases}}
\newcommand{\ECAS}{\end{cases}}
\newcommand{\BEAS}{\begin{eqnarray*}}
\newcommand{\EEAS}{\end{eqnarray*}}
\newcommand{\BEQ}{\begin{equation}}
\newcommand{\EEQ}{\end{equation}}
\newcommand{\BIT}{\begin{itemize}}
\newcommand{\EIT}{\end{itemize}}
\newcommand{\BMAT}{\begin{bmatrix}}
\newcommand{\EMAT}{\end{bmatrix}}
\newcommand{\BNUM}{\begin{enumerate}}
\newcommand{\ENUM}{\end{enumerate}}

\newcommand{\BA}{\begin{array}}
\newcommand{\EA}{\end{array}}

\newcommand{\reals}{\mathbb{R}}

\newcommand{\diag}{\mathop{\mathbf{diag}}}

\DeclareMathOperator*{\maximize}{maximize}

\DeclareMathOperator{\sign}{sign}

\newcommand{\norm}[1]{\left\| #1 \right\|}

\def\myF{{F}}

\def\hthp{\theta^{\rm p}}

\def\pe{{\rm p}}

\def\th{{\theta}}
\def\hth{{\widehat{\theta}}}

\def\pth{{\theta^{\pe}}}

\def\Dn{D}

\def\diag{{\rm diag}}

\def\event{\mathcal{E}}

\def\tZ{\tilde{Z}}

\def\G{{G}}
\def\dgamma{\widehat{\gamma}^{{\rm d}}}
\def\bE{\bar{\E}}

\newcommand{\ajcomment}[1]{}

\makeatletter
\newcommand{\labitem}[2]{%
\def\@itemlabel{\text{#1}}
\item
\def\@currentlabel{#1}\label{#2}}
\makeatother

\newcommand{\rev}[1]{{\color{black} #1}}
\addtocontents{toc}{\protect\setcounter{tocdepth}{2}}

\title{A Flexible Framework for Hypothesis Testing in High-dimensions}

\author{Adel~Javanmard\footnote{Data Sciences and Operations Department, University
of Southern California. Email: \url{ajavanma@usc.edu} }
             \;\; and\;\; 
Jason~D.~Lee\footnote{Department of Electrical Engineering, Princeton University. Email: \url{Jasonlee@princeton.edu}}
            }

\begin{document}
\maketitle

\begin{abstract}
Hypothesis testing in the linear regression model is a fundamental statistical problem.  We consider linear regression in the high-dimensional regime where the number of parameters exceeds
the number of samples ($p> n$). \rev{In order to make informative inference, we assume that the model is approximately sparse, that is the effect of covariates on the response can be well approximated by conditioning on a relatively small number of covariates whose identities are unknown.} We develop a framework for testing very general hypotheses regarding the model parameters.  Our framework encompasses testing whether the parameter lies in a convex cone, testing the signal strength, and testing arbitrary functionals of the parameter. We show that the proposed procedure controls the type I error 
, and also analyze the power of the procedure. Our numerical experiments confirm our theoretical findings and demonstrate that we control false positive rate (type I error) near the nominal level, and have high power. By duality between hypotheses testing and confidence intervals, the proposed framework can be used to obtain valid confidence intervals for various functionals of the model parameters.  For linear functionals, the length of confidence intervals is shown to be minimax rate optimal.
\end{abstract}
\section{Introduction}

Consider the high-dimensional regression model where we are given $n$ i.i.d. pairs $(y_1,x_1)$, $(y_2,x_2)$, $\cdots$,
$(y_n,x_n)$ with $y_i\in \reals$, and $x_i\in\reals^p$, denoting the response values and the feature vectors, respectively. The linear regression model posits that
response values are generated as
\begin{eqnarray}\label{eqn:regression}
y_i \,=\, \th_0^\sT x_i + w_i\, ,\;\;\;\;\;\;\;\; w_i\sim
\normal(0,\sigma^2)\, .
\end{eqnarray}
Here $\th_0 \in\reals^p$ is a vector of parameters to be estimated. 
In matrix form,
letting  $y = (y_1,\dots,y_n)^\sT$ and denoting by $X$ the matrix with
rows $x_1^\sT$,$\cdots$, $x_n^\sT$ we have
\begin{eqnarray}\label{eq:NoisyModel}
y\, =\, X\th_0+ w\, ,\;\;\;\;\;\;\;\; w\sim
\normal(0,\sigma^2 \id_{n\times n})\, .
\end{eqnarray}
%

We are interested in high-dimensional models where
the number of parameters $p$ may far exceed the sample size $n$. To make informative inference feasible in this setting, we assume sparsity structure for the model, that is $\theta_0$ has only a few ($s_0<n$) number of nonzero entries, whose identities are unknown.
  
Our goal in this paper is to understand various parameter structures of the high-dimensional model. Specifically, we develop a flexible framework for testing
null hypothesis of the form  
\begin{align}\label{eq:H0-HA}
H_0:\th_0\in \Omega_0\quad \text{ versus }\quad H_A: \theta_0\notin\Omega_0,,
\end{align}
for a general set $\Omega_0\subset \reals^p$. Remarkably, we make no additional assumptions (such as convexity or connectedness) on $\Omega_0$. 

\rev{
In Section~\ref{sec:app-sparsity}, we will relax the sparsity assumption on the model parameters to the \emph{approximate sparsity}. Consider the linear model $y = X\theta_* +w$, where $\theta_*\in \reals^p$ is not necessarily sparse. The approximate sparsity posits that even if the true signal $X\theta_*$ cannot be written as a sparse linear combination of the covariates, there exists at least one sparse linear combination of the covariates that gets close to the true signal. Formally, we assume that there exists a vector $\theta_0\in \reals^p$ such that $\|\theta_0\|_0 = s_0$, and $\|X\theta_* - X\theta_0\| = o_P(1)$. Note that this notion of approximate sparsity is similar to but stronger than the one introduced in~\cite{bunea2007sparsity, belloni2012sparse}.\footnote{In~\cite{belloni2012sparse}, the approximate sparsity assumption allows $\|X\theta_* - X\theta_0\| = O_P(\sqrt{s_0})$, while here we are imposing stronger requirement $\|X\theta_* - X\theta_0\| = o_P(1)$.  }

In addition, in Section~\ref{sec:non-gaussian} we extend our analysis  to non-gaussian noise. 
}

\subsection{Motivation}\label{sec:motivation}
High-dimensional models are ubiquitous in many areas of applications. Examples range from signal processing (e.g. compressed sensing), to
recommender systems (collaborative filtering), to statistical network
analysis, to predictive analytics, etc.
The widespread interest in these applications has spurred
remarkable progress in the area of high-dimensional data analysis \cite{Dantzig,BickelEtAl,buhlmann2011statistics}. Given that the number of parameters goes beyond the sample size, there is no hope to design reasonable estimators without
making further assumption on the structure of model parameters. A natural such assumption is sparsity, which posits that only $s_0$ of the parameters $\th_{0,i}$ are nonzero, and $s_0\le n$.
A prominent approach in this setting for estimating the model parameters is via the Lasso estimator \cite{Tibs96,BP95} defined by 
\begin{eqnarray}\label{eq:Lasso}
\Lth(y,X;\lambda)\equiv \arg\max_{\theta\in\reals^p}
\left\{\frac{1}{2n}\|y-X\theta\|_2^2+\lambda
\|\theta\|_1\right\}\,.
\end{eqnarray}
(We will omit the arguments of $\Lth(y,X;\lambda)$ whenever clear from the context.)

To date, the majority of work on high-dimensional parametric models has focused on point estimation
such as consistency for prediction~\cite{GreenshteinRitov}, oracle inequalities and estimation of parameter vector~\cite{Dantzig,BickelEtAl,WainwrightEllP},
model selection \cite{MeinshausenBuhlmann,zhao,Wainwright2009LASSO}, and variable screening~\cite{fan2008sure}. 
\rev{The work~\cite{bunea2007sparsity} extended the oracle inequalities for the lasso to the setting of weak sparsity and weak approximation, where the effect of covariates on the response can be controlled up to a small approximation error by conditioning on a relatively small number of covariates, whose identities are unknown. 
The minimax rate for estimating the parameters in the high-dimensional linear model was studied in~\cite{ye2010rate,raskutti2011minimax}, assuming that the true model parameters belong to some $\ell_q$ ball. }

Despite this remarkable progress, the fundamental problem of statistical
significance is far less understood in the high-dimensional setting. Uncertainty assessment
is particularly important when one seeks subtle statistical patterns about the model parameters $\th_0$.

Below, we discuss some important examples of high-dimensional inference that can be performed when provided a methodology for testing hypothesis of the form~\eqref{eq:H0-HA}. 
\bigskip

{\bf Example 1 (Testing $\th_{\min}$ condition)}
Support recovery in high-dimension concerns the problem of finding a set $\widehat{S} \subseteq \{1, 2,\dots,p\}$, such that $\prob(\widehat{S} = S)$ is large, where $S \equiv \{i: \, \th_{0,i}\neq 0,\, 1\le i\le p\}$. Work on support recovery requires the nonzero parameters be large enough to be detected. Specifically, for exact support recovery meaning that $\prob(\hat{S}\neq S) \to 1$, it is assumed that $\min_{i\in S} |\th_{0,i}|  = \Omega(\sqrt{(\log p)/n})$. This assumption is often referred to as $\th_{\min}$ condition and is shown to be necessary for exact support recovery~\cite{meinshausen2009lasso,zhao,fan2001variable,zhao,Wainwright2009LASSO,MeinshausenBuhlmann}.

Relaxing the task of exact support recovery, let $\alpha$ and $\beta$ be the type I and type II error rates in detecting nonzero (active) parameters of the model. In~\cite{javanmard2013hypothesis}, it is shown that even for gaussian design matrices, any hypothesis testing rule with nontrivial power $1-\beta >\alpha$ requires $\min_{i\in S} |\th_{0,i}| = \Omega(1/\sqrt{n})$.
Despite $\th_{\min}$ assumption is commonplace, it is not verifiable in practice and hence it calls for developing methodologies that can test whether such condition holds true.

For a vector $\th\in \reals^p$, define support of $\th$ as $\supp(\th) = \{1\le i\le p:\, \th_i\neq 0\}$. 
In~\eqref{eq:H0-HA}, letting $\Omega_0 = \{\th\in \reals^p:\, \min_{i\in \supp(\th)} |\th_i| \ge c\}$, we can test $\th_{\min}$ condition for any given $c\ge 0 $ and at a pre-assigned significance level $\alpha$. 

\bigskip

{\bf Example 2 (Confidence intervals for quadratic forms)}
We can apply our method to test hypothesis of form
\begin{align}
H_0:\;\|Q\th_0\|_2\in \Omega_0\,, 
\end{align}
for some given set $\Omega_0 \subseteq [0,\infty)$ and a given matrix $Q\in \reals^{m\times p}$. By duality between hypothesis testing and confidence interval, we can also construct confidence intervals for quadratic forms $\|Q\th_0\|$. 

In the case of $Q=\id$, this yields inference on the signal strength $\|\th\|_2^2$. 
As noted in~\cite{janson2016eigenprism}, armed with such testing method one can also provide confidence intervals for the estimation error, namely $\|\hth-\th_0\|_2^2$. Specifically, we split the collected samples into two independent groups $(y^{(0)},X^{(0)})$ and $(y^{(1)},X^{(1)})$, and construct an estimate $\hth$ just by using the first group. Letting $\tilde{y} \equiv y^{(1)}- X^{(1)}\hth$, we obtain a linear regression model $\tilde{y} = X^{(1)}(\th_0-\hth)+ w$. Further, if $\hth$ is a sparse estimate, then $\th_0-\hth$ is also sparse. Therefore, inference on the signal strength on the obtained model is similar to inference on the error size $\|\th_0-\hth\|_2^2$. 

Inference on quadratic forms turns out to be closely related to a number of well-studied problems, such as estimate of the noise level $\sigma^2$ and the proportion of explained variation~\cite{fan2012variance,bayati2013estimating,
dicker2014variance,janson2016eigenprism,verzelen2016adaptive,guo2017optimal}. To expand on this point, suppose that attributes $x_i$ are drawn i.i.d. from a gaussian distribution with covariance $\Sigma$, and the noise level $\sigma^2$ is unknown.
 Then, $\Var(y_i) = \sigma^2+\|\Sigma^{1/2} \th_0\|_2^2$. Since $\|y\|_2^2/\Var(y_i)$ follows a $\chi^2$ distribution with $n$ degrees of freedom, we have
$\|y\|_2^2/n = \Var(y_i)[1+O_P(n^{-1/2})]$. Hence, task of inference on the quadratic form $\|\Sigma^{1/2} \th_0\|_2^2$ and  the noise level $\sigma^2$ are intimately related.
This is also related to the proportion of explained variation defined as 
\begin{align}
 \eta(\th_0,\sigma) = \frac{\E((x_i^\sT \th_0)^2)}{\Var(y_i)} =\frac{\mu}{1+\mu}\,,
\end{align}
 with $\mu = (1/\sigma^2) \|\Sigma^{1/2}\th_0\|_2^2$ the signal-to-noise ratio. This quantity is of crucial importance in genetic variability~\cite{visscher2008heritability} as it somewhat quantifies the 
 proportion of variance in a trait (response) that is explained by genes (design matrix) rather than environment (noise part).

\bigskip

{\bf Example 3 (Testing individual parameters $\th_{0,i}$)}
Recently, there has been a significant interest in testing individual hypothesis $H_{0,i}: \th_i = 0$, in the high-dimensional regime.
This is a challenging problem because obtaining an exact characterization
of the probability distribution of the parameter estimates in the high-dimensional regime is notoriously hard.

A successful approach is based on debiasing the regularized estimators. The resulting debiased estimator is amenable to distributional characterization which can be used for inference on individual parameters~\cite{javanmard2014confidence,javanmard2013hypothesis,zhang2014confidence,van2014asymptotically,javanmard2013nearly}. Our methodology for testing hypothesis of form~\eqref{eq:H0-HA} is built upon the debiasing idea. It also recovers the debiasing approach for $\Omega_0 = \{\th\in \reals^p:\, \th_i = 0\}$.

\bigskip

{\bf Example 4 (Confidence intervals for predictions)}
For a new sample $\xi$, we can perform inference on the response value $\xi^\sT \th_0$ by letting $\Omega_0 = \{\th:\, \xi^\sT \th_0 = c\}$ for a given value $c$. Further, by duality between hypothesis testing and confidence intervals, we can construct confidence interval for $\xi^\sT \th_0$. We refer to Section~\ref{sec:Other} for further details.

\bigskip

{\bf Example 5 (Confidence intervals for $f(\th_0)$)}
Let $f:\reals^p \to\reals$ be an arbitrary function. By letting $\Omega_0 = \{\th:\, f(\th_0) = c\}$   we can test different values of $f(\th_0)$. Further, by employing the duality relationship between hypothesis testing and confidence intervals, we can construct confidence intervals for $f(\th_0)$. Note that Examples 3, 4 are special cases of $f(\th_0) = e_i^\sT \th_0$ and $f(\th_0) = \xi^\sT \th_0$. Here $e_i$ is the $i$-th standard basis element with one at the $i$-th entry and zero everywhere else.

\bigskip

{\bf Example 6 (Testing over convex cones)} For a given cone $\cC$, our framework allows us to test whether $\th_0$ belongs to $\cC$. Some examples that naturally arise in studying treatment effects are nonnegative cone $\cC_{\ge0} =\{\th\in \reals^p:\, \th_i\ge 0\,\,\text{for all } 1\le i\le p\}$, and monotone cone $\cC_M = \{\th\in\reals^p:\, \th_1\le \th_2\le \dotsc\le \th_p\}$. Letting $\th_i$ denote the mean of treatment $i$, by testing $\th_0\in \cC_{\ge0}$, one can test whether all the treatments in the study are harmless. Another case is when treatments correspond to an ordered set of dosages of the same drug. Then, one might reason that if the drug is of any effect, its effect should follow a monotone relationship with its dosage. This hypothesis can be cast as $\th_0\in \cC_M$.
Such testing problems over cones have been studied for gaussian sequence models by~\cite{kudo1963multivariate, robertson1978likelihood,raubertas1986hypothesis}, and very recently by~\cite{wei2017geometry}.

\subsection{Other Related work}
\label{sec:Related}
Testing in the high-dimensional linear model has experienced a resurgence in the past few years. Most closely related to us is the line of work on debiasing/desparsifying pioneered by \cite{zhang2014confidence,van2014asymptotically,javanmard2014confidence}. These papers propose a debiased estimator $\dth$ such that every coordinate $\dth _i$ is approximately gaussian under the condition that ${s_0^2(\log p) }/{n} \to 0$, which is in turn used to test single coordinates of $\th_0$, $H_{0}: \th_{0,i} =0$, and construct confidence intervals for $\th_{0,i}$. In a parallel line of work, \cite{belloni2011lasso,belloni2013program, belloni2011inference,belloni2014inference} have also designed an asymptotically gaussian pivot via the post-double-selection lasso, under the same sample size condition of ${s_0^2(\log p) }/{n} \to 0$. \cite{cai2015confidence} established that the sample size conditions required by debiasing and post-double-selection are minimax optimal meaning to construct a confidence interval of length $O({1}/{\sqrt{n}})$ for a coordinate of $\th_0$ requires ${s_0^2(\log p) }/{n} \to 0$. 

The debiasing and post-double-selection approaches have also been applied to a wide variety of other models for testing $\th_{0,i}$  including missing data linear regression \cite{wang2017rate}, quantile regression \cite{zhao2014general}, and graphical models \cite{ren2015asymptotic,chen2016asymptotically,wang2016inference,barber2015rocket}.

In the multiple testing realm, the debiasing approach has been used to control directional FDR~\cite{javanmard2018false}. Other methods such as FDR-thresholding and SLOPE procedures controls the false discovery rate (FDR) when the design matrix $X$ is orthogonal \cite{su2016slope,bogdan2015slope,abramovich2006special}. In the non-orthogonal setting, the knockoff procedure \cite{barber2014controlling} controls FDR whenever $n \ge 2p$, and the noise is isotropic; In \cite{janson2015familywise}, knockoff was generalized to also control for the family-wise error rate. More recently, \cite{candes2016panning} developed the model-free knockoff which allows for $p>n$ when the distribution of $X$ is known.

In parallel, there have been developments in selective inference, namely inference for the variables that the lasso selects. \cite{lee2016exact, tibshirani2016exact} developed exact tests for the regression coefficients corresponding to variables that lasso selects. This was further generalized to a wide variety of polyhedral model selection procedures including marginal screening and orthogonal matching pursuit in \cite{lee2014exact}. \cite{tian2015selective,fithian2014optimal,harris2016selective} developed more powerful and general selective inference procedures by introducing noise in the selection procedure. To allow for selective inference in the high-dimensional setting, \cite{lee2016exact} combined the polyhedral selection procedure with the debiased lasso to construct selectively valid confidence intervals for $\th_{0,i}$ when ${s_0(\log p)}/{\sqrt{n}} \to 0$. 

Much of the previous work has focused on testing coordinates or one-dimensional projections of $\th_0$. 
An exception is the work~\cite{nickl2013confidence} which studies the problem of constructing confidence sets for the high dimensional linear models, so that the confidence sets are honest over the family of sparse parameters, under i.i.d gaussian designs. 
Our work increases the applicability of the debiasing approach by allowing for general hypothesis, $\th_0 \in \Omega_0$. 
The set $\Omega_0$ can be non-convex or even disconnected. Our setup encompasses a broad range of testing problems and it is shown to be minimax optimal for special cases such as $\Omega = \{\th:\th_i  = 0\}$ and $\Omega_0 = \{\th:\xi^\sT \th = c\}$.

The authors in~\cite{zhu2017projection} have studied the problem~\eqref{eq:H0-HA} independently and indeed~\cite{zhu2017projection} was posted online around the same time that the first draft of our paper was released.
This work also leverages the idea of debiasing but greatly differs from this work, both in methodology and theory, which we now discuss. In~\cite{zhu2017projection}, the debiased estimator is constructed in the standard basis (as compared to ours which is done in a lower dimensional subspace) and is followed by an $\ell_1$ projection to construct the test statistic. The test statistic involves a data dependent vector and the method uses bootstrap to   approximate the distribution of the test statistic and set the critical values. In terms of theory,~\cite{zhu2017projection} shows that the proposed method controls the type I error at the desired level assuming that $\log p = o(n^{1/8})$ and $s_0 = o(n^{1/4}/\sqrt{\log p})$ (See Theorem 1 therein), while we prove such result for our test under $s_0 = o(\sqrt{n}/\log p)$. It is shown in~\cite{zhu2017projection} that the rule achieves asymptotic power one provided that the signal strength (measured in term of the $\ell_\infty$ distance of $\th_0$ from $\Omega_0$) asymptotically dominates $n^{-1/4}$. In comparison, in Theorem~\ref{thm:power} we establish a lower bound of the power for \emph{all values} of the signal strength and as a corollary of that we show the method achieves power one if the signal strength dominates $n^{-1/2}$ asymptotically.

\subsection{Organization of the paper}
In the remaining part of the introduction, we present the notations and a few preliminary definitions. 
The rest of the paper presents the following contributions: 

\begin{itemize}
\item Section~\ref{sec:test}. We explain our testing methodology. It consists of constructing a debiased  estimator for the projections of the model parameters in a lower dimensional subspace. It is then followed by an $\ell_\infty$ projection to form the test statistic.
\item Section~\ref{sec:Results}. We present our main results. Specifically, we show that our method controls false positive rate under a
pre-assigned $\alpha$ level. We also derive an analytical lower bound for the statistical power of our test. In case of $\Omega_0 = \{\th\in \reals^d:\, \th_{i} =0\}$ (Example 3), it matches the bound proposed in~\cite[Theorem 3.5]{javanmard2014confidence}, which is also shown to be minimax optimal.
\item Section~\ref{sec:app-sparsity}. We explain the notion of approximate sparsity and discuss how our results can be extended to allow for approximately sparse models.
\item Section~\ref{sec:non-gaussian}. We relax the gaussianity assumption on the noise component and discuss how to address possibly non-gaussian noise under proper moment conditions.
\item Section~\ref{sec:Other}. We provide applications of our framework for some special cases: Inference on linear predictions, quadratic forms of the parameters and testing the $\theta_{\min}$ condition. In Section~\ref{sec:prior-art}, we discuss the existing literature for these subproblems and compare it to our proposed methodology.
\item Section~\ref{sec:numerical}.  We provide numerical experiments to corroborate our findings and evaluate type I error and statistical power of our test under various settings.  
\item Section~\ref{proof:theorems}. Proof of Theorems are given in this section, while the proof of technical lemmas are deferred to appendices.
%
\end{itemize}

%
%

\subsection{Notations}\label{sec:preliminary}
We start by adapting some simple notations that will be used throughout the paper, along with some basic definitions from the literature on high-dimensional regression.

We use $e_i$ to
refer to the $i$-th standard basis element, e.g., $e_1 = (1,0,\dotsc,0)$. For a vector $v$, $\supp(v)$ represents
the positions of nonzero entries of $v$. For a vector $\theta$ and a subset $S$, 
$\theta_S$ is the restriction of $\theta$ to indices in $S$. For an integer $p\ge 1$, we use the notation $[p] = \{1,\cdots, p\}$. We write $\|v\|_p$ for the standard $\ell_p$ norm of a vector $v$, i.e., $\|v\|_p = (\sum_i |v_i|^p)^{1/p}$ and $\|v\|_0$ for the number of nonzero entries of $v$.  Whenever the subscript $p$ is not mentioned it should be read as $\ell_2$ norm. 
For a matrix $A$, we denote by $|A|_{\infty}
\equiv\max_{i\le m, j\le n}|A_{ij}|$, the maximum absolute value of entries of $A$.
Further, its maximum and minimum singular values are respectively indicated by by $\sigma_{\max}(A)$ and $\sigma_{\min}(A)$.
Throughout, $\Phi(x) \equiv \int_{-\infty}^x e^{-t^2/2} \de t/\sqrt{2\pi}$ denotes
the CDF of the standard normal distribution. We also denote the $z$-values $z_{\alpha} = \Phi^{-1}(1-\alpha)$.

The term ``with high probability" means with probability converging to one as $n \to \infty$ and for two
functions $f(n)$ and $g(n)$, the notation $f(n) = o(g(n))$ means that $g$ `dominates' $f$ asymptotically,
namely, for every fixed positive $C$, there exists $n(C)$ such that $f(n) \le C g(n)$ for $n > n(C)$. 
Likewise, $f(n) = O(g(n))$ indicates that $f$ is `bounded' above by $g$ asymptotically, i.e., $f(n)\le C g(n)$ for some positive constant $C$. 
Analogously, we use he notations $o_P(\cdot)$ and $O_P(\cdot)$ to indicate asymptotic behavior is probability as the sample size $n$grows to infinity.  

 Let $\hSigma = (X^\sT X)/n \in \reals^{p\times p}$ be the sample covariance of the design $X\in \reals^{n\times p}$. In the high-dimensional setting, where $p$ exceeds $n$, $\hSigma$
 is singular. As common in high-dimensional statistics, we assume \emph{compatibility condition} which requires $\hSigma$ to be nonsingular in a restricted set of directions.
 
 We use the notation $\|\cdot\|_{\psi_2}$ to refer to the sub-gaussian norm. Specifically, for a random variable $X$,  we let 
 \begin{align}\label{def:subG}
 \|X\|_{\psi_2} = \sup_{q\ge 1} q^{-1/2} (\E|X|^q)^{1/q}\,.
 \end{align}
 
For a random vector $X\in\reals^m$, its sub-gaussian norm is defined as 
$$\|X\|_{\psi_2} = \sup_{\|x\|\le 1} \|\<X,x\>\|_{\psi_2}\,.$$

 \begin{definition}\label{eq:compatibility}
 For a symmetric matrix $J\in \reals^{p\times p}$ and a set $S\subseteq [p]$, the compatibility condition is defined as 
 \begin{align}
 \phi^2(J, S) \equiv \min_{\th\in \reals^p} \Big\{\frac{|S| \<\th, J\th\>}{\|\th_S\|_1^2}:\quad \th\in \reals^p,\, \|\th_{S^c}\|_1\le 3\|\th_S\|_1 \Big\}\,.
 \end{align}
 Matrix $J$ is said to satisfy compatibility condition for a set $S\subseteq [p]$, with constant $\phi_0$ if $\phi(J,S) \ge \phi_0$.
 \end{definition}   

\section{Projection statistic}\label{sec:test}
Depending on the structure of $\Omega_0$ it may be useful to instead of testing the null hypothesis $H_0: \theta_0 \in \Omega_0$, we test it in a lower dimensional space. 
Consider an $k$-dimensional subspace represented by an orthonormal basis $\{u_1,\dotsc, u_k\}$, with $u_i\in \reals^p$. For this section, we assume that the basis $\{ u_1, \ldots, u_k\}$ is predetermined and fixed. In Section \ref{sec:choice-U}, we discuss how to choose the subspace depending on $\Omega_0$ to maximize the power of the test. The projection onto this subspace is given by
\[
\proj_U(\theta) = \sum_{i=1}^k \<\theta,u_i\> u_i = UU^\sT \theta\,,
\]
where $U = [u_1, \dotsc, u_k]\in\reals^{p\times k}$.  We also use the notation $\proj_U(\Omega_0) = \{\proj_U(\theta):\,\theta\in \Omega_0\}$ to denote the projection of $\Omega_0$ onto the subspace $U$. 
Define the hypothesis
\begin{align}
\tilde{H}_0: \proj_U(\theta_0)\in \proj_U(\Omega_0)\,.
\end{align}
Under the null $H_0$, $\tilde H_0$ also holds, so controlling the type-I error of $\tilde H_0$ also controls the type-I error of $H_0$.  In the following we propose a testing rule $R \in \{0,1\}$ for the null hypothesis $\tilde{H}_0$ and show that it controls type-I error below a pre-assigned level $\alpha$. Consequently,
\[
\sup_{\theta\in \Omega_0} \prob_{\theta}(R = 1) \le \sup_{\proj_U(\theta)\in \proj_U(\Omega_0)} \prob_\theta(R = 1) \le \alpha\,.
\] 
For now, we consider an arbitrary fixed subspace $U$, and then after we analyze the statistical power of our test we provide guidelines on how to choose $U$ to increase the power. 

In order to test $\tilde{H}_0$ we construct a test statistic based on the debiasing approach. 

We first let $\{\hth,\hsigma\}$ be the scaled Lasso estimator~\cite{SZ-scaledLasso} given by
\begin{align}\label{scaledLasso}
\{\Lth(\lambda),\hsigma(\lambda)\} = \underset{\th\in \reals^p,\sigma>0}{\arg\min} \left\{\frac{1}{2\sigma n} \|y-X\th\|_2^2+\frac{\sigma}{2} + \lambda \|\th\|_1 \right\}\,.
\end{align}
This optimization simultaneously gives an estimate of $\th_0$ and $\sigma$. We use regularization parameter $\lambda = \sqrt{2.05 (\log p)/n}$. 
Due to the $\ell_1$ penalization, the lasso estimator $\hth$ is biased towards small $\ell_1$ norm, and so is the projection $\proj_U(\theta_0)$.
We view $\proj_U(\theta_0)$ in the basis $U$, namely $\gamma_0 = U^\sT \theta_0$ and construct a debiased estimator for it in the following way:
\begin{align}\label{dgamma}
\dgamma = U^\sT \htheta + \frac{1}{n} \G^\sT X^\sT(y-X\htheta)\,,
\end{align}
with the decorrelating matrix $\G = [g_1|\dotsc|g_k]\in \reals^{p\times k}$, where each $g_i$ is obtained by solving the optimization problems for each $1\le i\le k$:
%
\begin{align}\label{OPT:M-n}
\begin{split}
&\text{minimize} \quad g^\sT \hSigma g\\
&\text{subject to} \;\; \|\hSigma g - u_i\|_\infty \le \mu
\end{split}
\end{align}
Note that the decorrelating matrix $G\in\reals^{p\times p}$ is a function of $X$, but not
of $y$. We next state a lemma that provides a a bias-variance decomposition for $\dgamma$ and brings insight about the form of debiasing given by~\eqref{dgamma}.
%
%
%


\begin{lemma}\label{propo:bias-size2} 
Let $X\in \reals^{n\times p}$ be any (deterministic) design matrix. Assuming that optimization problem~\eqref{OPT:M-n} is feasible for $i\in [k]$, let $\dgamma =\dgamma(\lambda)$ be a 
general debiased estimator as per Eq~\eqref{dgamma}. 
Then, setting $Z = G^\sT X^\sT w/\sqrt{n}$, with $w$ the noise vector in the regression~\eqref{eq:NoisyModel}, we have
\begin{align}\label{dgamma-decomposition}
\sqrt{n}(\dgamma-U^\sT \th_0) = Z+\Delta\,,\quad Z\sim\normal(0,\sigma^2 G^\sT\hSigma G)\,,\quad \Delta = \sqrt{n}(G^\sT \hSigma - U^\sT)(\th_0-\hth)\,.
\end{align} 
Further, assume that $X$ satisfies the compatibly condition for the set $S = \supp(\theta_0)$, $|S|\le s_0$, with constant $\phi_0$, and let $K\equiv \max_{i\in[p]}(X^\sT X/n)_{ii}$. Then, choosing $\lambda = c \sqrt{(\log p)/n}$, we have
\begin{align}
\prob\left(\|\Delta\|_\infty \ge  \frac{c\mu\sigma s_0}{\phi_0^2}\sqrt{\log p} \right)\le 2p^{-c_0}+ 2e^{-n/16}\,, \quad c_0 = \frac{c^2}{32K}-1\,.
\label{eq:delta-small}
\end{align}
\end{lemma}

Lemma~\ref{propo:bias-size2} can be proved in a similar way to Theorem 2.3 of \cite{javanmard2014confidence} and its proof is omitted here. The decomposition~\eqref{dgamma-decomposition} explains the rationale behind  optimization~\eqref{OPT:M-n}. Indeed the convex program~\eqref{OPT:M-n} aims at optimizing two objectives. On one hand, the constraint controls the term $|G^\sT\hSigma - U^\sT|_\infty$, which by Lemma~\ref{propo:bias-size2} controls the bias term $\|\Delta\|_\infty$. On the other hand, it minimizes the objective function $g^\sT \hSigma g$, which controls the variance of $\dgamma_i$.  Therefore, the parameter $\mu$ in optimization~\eqref{OPT:M-n} controls the bias-variance tradeoff and should be chosen large enough to ensure that~\eqref{OPT:M-n} is feasible. 
(See Section~\ref{mu*} for further discussion.) 
\begin{remark}
In the special case of $k =1$ and $u = e_i$, the debiased estimator~\eqref{dgamma} reduces to the one introduced in~\cite{javanmard2014confidence}. For the special case of $k =1$, it becomes similar to the estimator proposed by~\cite{cai2015confidence} that is used to construct confidence intervals for linear functionals of $\theta_0$.  Note that
the proposed debiasing procedure incurs small bias in the infinity norm with respect to the rotated basis, $\norm{\dgamma - U^\sT\theta_0}_\infty$, as opposed to the standard debiasing procedure~ \cite{javanmard2014confidence,javanmard2013hypothesis,zhang2014confidence,van2014asymptotically,javanmard2013nearly} which incurs small bias, in the infinity norm, with respect to the original basis, and not necessarily in the rotated basis.
\end{remark}
\rev{
The following assumption ensures that the entries of noise $Z$ have non-vanishing variances.
\begin{assumption}\label{ass:var}
We have $\lim\inf_{n\to\infty} \min_{i\in [k]} (G^\sT \hSigma G)_{i,i} \ge c_0>0$, for some positive constant $c_0$.
\end{assumption}
The above assumption entails the decorrelating matrix $G$, where our proposal constructs via optimization~\eqref{OPT:M-n}. In the following lemma, we provide a sufficient condition for the above assumption to hold.
\begin{lemma}\label{lem:sufficient}
Suppose that $\lim \sup_{n\to \infty}  \mu  (\max_{i\in [k]}\, \|u_i\|_1)\le c<1$ and $\lim \sup_{n\to\infty} \max_{i\in [k]} (u_i^\sT \hSigma u_i) < C<\infty$, for some constant $c, C$.  Then, Assumption~\ref{ass:var} holds.
\end{lemma} 
We refer to Appendix~\ref{app:lem-sufficient} for the proof of Lemma~\ref{lem:sufficient}.

\begin{remark}
The very recent work~\cite{cai2019individualized} uses the debiasing approach for inference on individualized treatment effect (and for general linear function $u^T\theta_0$ ). The proposed mechanism slightly differs from~\eqref{OPT:M-n} in that it includes an extra constraint. By this trick, the proposed mechanism of~\cite{cai2019individualized} can be used for inference on a broad family of loading vector $u$.
We can follow the same idea and replace optimization~\eqref{OPT:M-n} by the following optimization
\begin{align}\label{OPT:M-n-update}
\begin{split}
&\text{minimize} \quad g^\sT \hSigma g\\
&\text{subject to} \;\; \|\hSigma g - u_i\|_\infty \le \mu\,,\\
&\quad \quad\quad \quad \quad |u_i^\sT \hSigma g - 1|\le \mu\,.
\end{split}
\end{align}
This way Assumption~\ref{ass:var} is automatically satisfied (See~\cite[Lemma 1]{cai2019individualized} for the details). 
\end{remark}
}

%
%
%
Define the shorthand
\begin{align}\label{Q}
Q^{(n)} \equiv \frac{\hsigma^2}{n} (\G^\sT\hSigma \G)\,,\quad D^{(n)} \equiv \diag(\{Q^{(n)}_{ii}\}^{-1/2})\,.
\end{align}
To ease the notation, we hereafter drop the superscript $(n)$.
We next construct a test statistic $T_n$ so that the large values of $T_n$ provide evidence against the null hypothesis.  For this, consider the $\ell_\infty$ projection
 estimator given by 
\begin{align}\label{OPT1-n}
\begin{split}
\hthp = & \quad \underset{\theta \in \reals^p}{\text{argmin}} \quad \quad \|\Dn(\dgamma - U^\sT \theta)\|_\infty\\
&\quad \text{subject to} \;\;\; \th\in \Omega_0\,. 
\end{split}
\end{align}
We then define the test statistic to be the optimal value of~\eqref{OPT1-n}, i.e., 
\begin{align}\label{test-statistic}
T_n = \|\Dn(\dgamma-U^\sT \hthp)\|_\infty
\end{align}
The reason for using $\ell_\infty$ norm in the projection is that the bias term of $\dgamma$ is controlled in $\ell_\infty$ norm (See Lemma~\ref{propo:bias-size2}.) 
The decision rule is then based on the test statistic: 
\begin{align}\label{detection}
R_X(y) = \begin{cases}
1 & \text{ if } T_n \ge z_{\alpha/(2k)} \quad \quad (\text{reject }\tilde{H}_0)\\
\\
0 & \text{ otherwise } \quad \quad \quad \, \,\, (\text{fail to reject }\tilde{H}_0).
\end{cases}
\end{align}

%
%

The above procedure generalizes the debiasing approach of \cite{javanmard2014confidence}. Specifically, for $\Omega_0 =\{\th : \theta_1=0\}=\{0\} \times \reals^{p-1}$ and $U = e_1 e_1^\sT$, the test rule becomes the one proposed by~\cite{javanmard2014confidence} for testing hypothesis of the form $H_{0}: \theta_{0,1} = 0$ versus its alternative. 

\rev{
\begin{remark}
Using Lemma~\ref{propo:bias-size2}, under the null hypothesis $H_0:\,\theta_0\in \Omega_0$, we have that $D(\dgamma - U^\sT \pth)$ is (asymptotically) stochastically dominated by $DZ$, whose entries are dependent and are distributed as standard normal. The choice of threshold $z_{\alpha/(2k)}$ in~\eqref{detection} comes from using this observation and union bounding to control the (two-sided) tail of $\|DZ\|_\infty$. Given that Lemma~\ref{propo:bias-size2} also characterizes the dependency structure of the entries of $DZ$, we can pursue another (less conservative) approach to choose the rejection threshold.  As an implication of  Lemma~\ref{propo:bias-size2}, and since $k$ (dimension of $Z$) is fixed, we have that for all $t\in \reals$, 
\begin{align}\label{eq:asymp-dist}
\prob\left(\|D(\dgamma - U^\sT \theta_0)\|_\infty \le t\right) - \prob\left(\|D Z\|_\infty \le t\right) = o_P(1)\,.
\end{align}
Under the null hypothesis $H_0$, we have $\|D(\dgamma - U^\sT \pth)\|_\infty \le \|D(\dgamma - U^\sT \theta_0)\|_\infty$, and by~\eqref{eq:asymp-dist}, the distribution of $\|D(\dgamma - U^\sT \theta_0)\|_\infty$ is asymptotically equal to the maximum of dependent standard normal variables $\|DZ\|_\infty$, whose distribution can be easily simulated since the covariance of the multivariate gaussian vector $DZ$ is known.
\end{remark}
}


In the next section, we prove that decision rule~\eqref{detection} controls type-I error below the target level $\alpha$ provided the basis $U$ is independent of the samples $(y_i,x_i)$, $1\le i\le n$. 
We also develop a lower bound on the statistical power of the testing rule and use that to choose the basis $U$.

\section{Main results}
\label{sec:Results}

\subsection{Controlling false positive rate}

%
\begin{definition}\label{mu*} 
Consider a given triple $(X;U;\G)$ where $X\in \reals^{n\times p}$, $U \in\reals^{p\times k}$ with $U^\sT U = I$ and $\G\in \reals^{p\times k}$.
The \emph{generalized coherence} parameter of $(X; U;\G)$ denoted by $\mu_*(X;U;\G)$ is given by
\begin{align}\label{incoherence}
\mu_*(X;U;\G) \equiv |\hSigma\G - U|_\infty\,,
\end{align} 
where $\hSigma = (X^\sT X)/n$ is the sample covariance of $X$.
The minimum generalized coherence of $(X;U)$ is $\mu_{\min}(X;U) = \min_{\G\in \reals^{p\times k}} \mu_*(X; U;\G)$. 
\end{definition}
Note that choosing $\mu\ge \mu_{\min}(X;U)$, the optimization~\eqref{OPT:M-n} becomes feasible.

We take a minimax perspective and require that the probability of type I error (false positive) to be controlled uniformly over $s_0$-sparse vectors.

For a testing rule $R\in \{0,1\}$ and a set $\Omega_0$, we define
\begin{align}
\alpha_n(R) \equiv \sup \Big\{\prob_{\th_0}(R = 1):\,\, \th_0\in \Omega_0,\; \|\th_0\|_0 \le s_0(n) \Big\}\,.
\end{align}

Our first result shows validity of our test for general set $\Omega_0$ under deterministic designs.

\begin{thm}\label{thm:typeI}
Consider a sequence of design matrices $X\in\reals^{n\times p}$, with dimensions $n\to \infty$, $p = p(n) \to \infty$ satisfying the following assumptions. 
For each $n$, the sample covariance $\hSigma  = (X^\sT X)/n$ satisfies compatibility condition for the set
$S_0 = \supp(\th_0)$, with a constant $\phi_0>0$.  Also, assume that $K\ge \max_{i\in[p]} \hSigma_{ii}$ for some constant $K >0$. Also consider a sequence of matrices $U\in \reals^{p\times k}$, with fixed $k$ and $p = p(n)\to \infty$, such that $U^\sT U = I_k$. 

Consider the linear regression~\eqref{eq:NoisyModel} and let $\Lth$ and $\hsigma$ be obtained by scaled Lasso, given by~\eqref{scaledLasso},
with $\lambda = c\sqrt{(\log p)/n}$. Construct a debiased estimator $\dgamma$ as in~\eqref{dgamma} using 
$\mu \ge \mu_{\min}(X;U)$, where $\mu_{\min}(X;U)$ is the minimum generalized coherence parameter as per Definition~\ref{mu*}, and suppose that Assumption \ref{ass:var} holds. Choose $c^2 > 32K$
 and suppose that $s_0 = o(\min\{{1}/(\mu\sqrt{\log p}),{n}/{\log p}\})$.
For the test $R_X$ defined in Equation~\eqref{detection}, and for any $\alpha \in [0,1]$, we have
\begin{align}
\underset{n\to \infty}{\lim\sup}\;\; \alpha_n(R_X) \le \alpha\,.
\end{align}
\end{thm}

\bigskip

We next prove validity of our test for general set $\Omega_0$ under random designs.

\begin{thm}\label{thm:typeI-random}
Let $\Sigma\in \reals^{p\times p}$ such that $\sigma_{\min}(\Sigma) \ge C_{\min}>0$ and $\sigma_{\max}(\Sigma) \le C_{\max} <\infty$ and $\max_{i\in[p]} \Sigma_{ii}\le 1$. 
Suppose that $X\Sigma^{-1/2}$ has independent sub-gaussian rows, with mean zero and sub-gaussian norm $\|\Sigma^{-1/2} x_1\|_{\psi_2} = \kappa$, for some constant $\kappa>0$.

Let $\Lth$ and $\hsigma$ be obtained by scaled Lasso, given by~\eqref{scaledLasso},
with $\lambda = c\sqrt{(\log p)/n}$, and $c^2 > 48$. Consider an arbitrary $U\in \reals^{p\times k}$, with $U^\sT U = I$, that is independent of the samples $\{(x_i,y_i)\}_{i=1}^n$. Construct a debiased estimator $\dgamma$ as in~\eqref{dgamma} with 
$\mu =a\sqrt{(\log p)/n}$ and $a^2> 48e^2\kappa^4C_{\max}/C_{\min}$. 
In addition, suppose that $\lim \sup_{n\to \infty}  \mu  (\max_{i\in [k]}\, \|u_i\|_1)\le c'$, for some constant $0<c'<1$ and  $s_0 = o(\sqrt{n}/{\log p})$.

For the test $R_X$ defined in Equation~\eqref{detection}, and for any $\alpha \in [0,1]$, we have
\begin{align}
\underset{n\to \infty}{\lim\sup}\;\; \alpha_n(R_X) \le \alpha\,.
\end{align}
\end{thm}

We refer to Section~\ref{proof:theorems} for the proof of Theorem~\ref{thm:typeI} and~\ref{thm:typeI-random}.

%
 
\subsection{Statistical power}
We next analyze the statistical power of our test. Before proceeding, note that without further assumption, we cannot achieve any non-trivial power, namely, power of $\alpha$ which is obtained by a rule that randomly rejects null hypothesis with probability $\alpha$. Indeed, by choosing $\th_0\notin \Omega_0$ but arbitrarily close to $\Omega_0$, once can make $H_0$ essentially indistinguishable from $H_A$. 
Taking this point into account, for a set $\Omega_0 \subseteq \reals^p$ and $\th_0\in \reals^p$, we define the distance $\de(\th_0,\Omega_0)$ as  
\begin{align}
\de(\th_0,\Omega_0;U) = \inf_{\theta\in \Omega_0} \|U^\sT(\th - \theta_0)\|_\infty\,.
\end{align}
We will assume that, under alternative hypothesis, $\de(\th_0,\Omega_0;U) \ge \eta$ as well. Define
\begin{align}
\beta_n(R) \equiv \sup\Big\{\prob_{\th_0}(R = 0):\,\, \|\th_0\|_0\le s_0(n), \; \de(\th_0,\Omega_0;U) \ge \eta\Big\}
\end{align}

Quantity $\beta_n$ is the probability of type II error (false negative) and $1-\beta_n$ is the statistical power of the test.

\begin{thm}\label{thm:power}
Let $R_X$ be the test defined in Equation~\eqref{detection}. Under the conditions of Theorem~\ref{thm:typeI-random}, for all $\alpha \in [0,1]$:
\begin{align}
\underset{n\to \infty}{\lim\inf}\;\; \frac{1-\beta_n(R_X)}{1-\beta_n^*(\eta)} \ge 1\,, \quad 1-\beta_n^*(\eta) \equiv \myF\left({\alpha},\frac{\sqrt{n}\eta}{\hsigma m_0}, k\right)_+
\end{align}
where we define $m_0$ as 
\begin{align}\label{def:m0}
m_0 \equiv \max_{i\in [k]}\, (u_i^\sT \Sigma^{-1} u_i )^{1/2}\,.
\end{align} 
Further, for $\alpha \in [0,1]$, $x\in \reals_+$, and integer $k \ge 1$, the function $F(\alpha, x, k)$ is defined as follows:
\begin{align}\label{eq:F-def}
\myF(\alpha,x,k ) = 1 - k\Big\{\Phi\left(x+\Phi^{-1}\left(1-\frac{\alpha}{2k}\right) \right) - \Phi\left(x - \Phi^{-1}\left(1-\frac{\alpha}{2k}\right) \right) \Big\}\,.
\end{align}
\end{thm}
The proof of Theorem~\ref{thm:power} is given in Section~\ref{proof:power}.

 Note that for any fixed $k\ge1$ and $\alpha>0$, the function $x\mapsto \myF(\alpha, x, k)$ is continuous and monotone increasing, i.e., the larger $\de(\th_0,\Omega_0)$ the higher power is achieved. Also,
 in order to achieve a specific power $\beta>\alpha$, our scheme requires $\eta>c_\beta m_0 (\sigma/\sqrt{n}) $, for some
constant $c_\beta$ that depends on  the desired power $\beta$. In addition, if $\eta \sqrt{n}\to \infty$, the rule achieves asymptotic power one.
 
 It is worth noting that in case of testing individual parameters $H_{0,i}: \th_{0,i} = 0$ (corresponding to $\Omega_0 = \{\th\in \reals^p:\, \th_{i} = 0\}$ and $k=1$), we recover the power lower bound established in~\cite{javanmard2014confidence}, which by comparing to the minimax trade-off studied in~\cite{javanmard2013hypothesis}, is optimal up to a constant. 
 
\section{Choice of subspace $U$}
\label{sec:choice-U}
Before we start this section, let us stress again that by Theorems~\ref{thm:typeI} and~\ref{thm:typeI-random}, the proposed testing rule controls type-I error below the desired level $\alpha$, \emph{for any choice of $U\in \reals^{p\times k}$, with $1\le k\le p$ and $U^\sT U = \id$ that is independent of $X$}. Here, we provide guidelines for choosing $U$ that yields high power. To this end we use the result of Theorem~\ref{thm:power}.

Note that
\[
m_0\le \max_{i\in [k]} (C_{\min}^{-1} \|u_i\|^2 )^{1/2} = C_{\min}^{-1/2} \,,
\]
where we recall that $\sigma_{\min}(\Sigma) > C_{\min}>0$ and $\|u_i\| = 1$, for $i\in [k]$. 
Hence,
\begin{align}\label{power-LB1}
\myF\left({\alpha},\frac{\sqrt{n}\,\de(\th_0,\Omega_0;U)}{\hsigma m_0}, k\right) \ge \myF\left({\alpha},\frac{1}{\hsigma} \sqrt{n C_{\min}}\,\de(\th_0,\Omega_0;U), k\right)\,.
\end{align}  
We propose to choose $U$ by maximizing the right-hand side of~\eqref{power-LB1}, which by Theorem~\ref{thm:power} serves as a lower bound for the power of the test. 
%
%
Nevertheless, the above optimization involves $\theta_0$ which is unknown. To cope with this issue, we use the Lasso estimate $\htheta$ via the following procedure:
\begin{enumerate}
 \item We randomly split the data $(y,X)$ into two subsamples $(y^{(1)}, X^{(1)})$ and $(y^{(2)}, X^{(2)})$ each with sample size $n_0 = n/2$. We let $\hth^{(1)}$ be the optimizer of the scaled Lasso applied to $(y^{(1)},X^{(1)})$.
 \item We choose $U\in \reals^{p\times k}$ by solving the following optimization:
 \begin{align}\label{opt:final0}
 \underset{k \in [p],U\in \reals^{p\times k}, U^\sT U = \id}{\text{maximize}} \,\myF\left({\alpha},\frac{1}{\hsigma} \sqrt{n C_{\min}}\,\de(\th_0,\Omega_0;U), k\right) \,.
 \end{align} 
 \item We construct the debiased estimator using the data $(y^{(2)},X^{(2)})$. Specifically, set $\hSigma^{(2)} \equiv (1/n_0) (X^{(2)})^\sT (X^{(2)})$ and let $g_i$  be the solution of the following optimization problems for each $1\le i\le k$:
 \begin{align}
\begin{split}
&\text{minimize} \quad g^\sT \hSigma^{(2)} g\\
&\text{subject to} \;\; \|\hSigma^{(2)} g - u_i\|_\infty \le \mu
\end{split}
\end{align}
Define the decorrelating matrix $\G = [g_1|\dotsc|g_k]\in \reals^{p\times k}$ and let $\hth^{(2)}$ be the optimizer of the scaled Lasso applied to $(y^{(2)},X^{(2)})$.  Let
\begin{align}\label{dgamma-simple}
\dgamma = U^\sT \hth^{(2)} + \frac{1}{n_0} G^\sT (X^{(2)})^\sT (y^{(2)} - X^{(2)}\hth^{(2)})\,.
\end{align}
\item Set $Q\equiv (\hsigma^2/n) (G^\sT \hSigma^{(2)} G)$ and $D \equiv \diag(\{Q_{ii}\}^{-1/2})$. Find the $\ell_\infty$ projection as
\begin{align}\label{my-Tn}
\begin{split}
\hthp = & \quad \underset{\theta \in \reals^p}{\text{argmin}} \quad  \|D (\dgamma - U^\sT \theta)\|_\infty\,
\quad \text{subject to} \;\;\; \th\in \Omega_0\,. 
\end{split}
\end{align}
\item Define the test statistics $T_n = \|D(\dgamma-U^\sT \hthp)\|_\infty$. The testing rule is given by
\begin{align}\label{detection2}
R_X(y) = \begin{cases}
1 & \text{ if } T_n \ge z_{\alpha/(2k)} \quad \quad (\text{reject }{H}_0)\\
0 & \text{ otherwise } \quad \quad \; \; (\text{fail to reject }{H}_0).
\end{cases}
\end{align}
 \end{enumerate}
 
 Note that the data splitting above ensures that $U $ is independent of $(y^{(2)},X^{(2)})$, which is required for our analysis (See Theorems~\ref{thm:typeI}, \ref{thm:typeI-random} and \ref{thm:power}.) 
\subsection{Convex sets $\Omega_0$}\label{sec:convex}
When the set $\Omega_0$ is convex, step (2) in the above procedure can be greatly simplified. Indeed, we can only focus on $k = 1$ in this case. 

\begin{lemma}\label{lem:k1}
Define the set $\cJ$ of matrices as 
\begin{align}
\cJ\equiv \arg\max_{U\in \reals^{p\times k}} \myF\left(\alpha,\frac{1}{\hsigma} \sqrt{n C_{\min}}\, \de(\hth^{(1)},\Omega_0;U), k\right) \quad \text{ subject to }\quad  1\le k\le p, \;\; U^\sT U = \id_k\,. 
\end{align} If $\Omega_0$ is convex then there exists a unit norm $u^*\in \reals^{p\times 1}$ such that $u^*\in \cJ$.  
\end{lemma}
Proof of Lemma~\ref{lem:k1} is given in Appendix~\ref{proof:lem:k1}.

Focusing on $k=1$, optimization~\eqref{opt:final0} reduces to the following optimization over $u\in \reals^{p\times 1}$:
\begin{align}
u\in \arg\max_{u\in \reals^{p}, \|u\|_2 = 1} \myF\left(\alpha,\frac{1}{\hsigma} \sqrt{n C_{\min}}\, \de(\hth^{(1)},\Omega_0;u), 1\right) \,.
\end{align}
The function $x\mapsto \myF(\alpha,x,k)$ is monotone increasing in $x$ and by substituting for $\de(\theta_0,\Omega_0;u)$, this becomes equivalent to the following problem:
\begin{align}
\underset{u\in \reals^{p}, \|u\|_2 \le 1}{\text{maximize}} \, \inf_{\theta\in \Omega_0} |u^\sT (\theta - \hth^{(1)})| \,.
\end{align}
Given that the objective is linear in $u$ and $\theta$, and the set $\Omega_0$ is convex we can apply the Von Neumann's minimax theorem and change the order of $\max$ and $\min$:
\begin{align}
\inf_{\theta\in \Omega_0} \underset{u\in \reals^{p}, \|u\|_2 \le 1}{\max} \,  |u^\sT (\theta - \hth^{(1)})| \,.
\end{align}
Denote the orthogonal projection of $\hth^{(1)}$ onto $\Omega_0$ by $\proj_{\Omega_0}(\hth^{(1)}) = \arg\min_{\th\in\Omega_0} \|\th - \hth^{(1)}\|_2$. Then it is straightforward to see that the optimal $u$ is given by 
\begin{align}\label{eq:u-opt}
u = \frac{\proj^\perp_{\Omega_0}(\hth^{(1)}) }{\|\proj^\perp_{\Omega_0}(\hth^{(1)}) \|}\,,
\end{align}
with $\proj^\perp_{\Omega_0}(\hth^{(1)}) = \hth^{(1)} - \proj_{\Omega_0}(\hth^{(1)})$.

We remind again that the type I error is controlled at the desired level for any $U\in \reals^{p\times k}$ with $U^\sT U = \id$ that is independent of $(y,X)$. The choice of $u$ in~\eqref{eq:u-opt} is a guideline for increasing power in case of convex sets $\Omega_0$.

\rev{
\begin{remark}\label{rem:example-nonCvx}
Let us stress again that convexity assumption of set $\Omega_0$ is crucial in deriving the recipe~\eqref{eq:u-opt}. To build further insight, we provide a concert example of a non-convex $\Omega_0$ and argue that $k=1$ is not the right choice.   Let $\Omega_0 = \Omega_1\cup \Omega_2$, where $\Omega_i = \{x\in \reals^p:\, |x_i|\le a, \, |x_j|\le 3a, \, \text{ for } j\neq i\}$, for $i=1, 2$ and a fixed constant $a>0$. Let $\theta_0 = (2a, 2a, 0,\dotsc, 0)\in \reals^p$. Observe that $\Omega_0$ is not convex and $\theta_0 \notin \Omega_0$. By choosing $k= p$ and $U = \id_{p\times p}$, we have $\de(\theta_0, \Omega_0,U) = a$ and hence our method achieves non-trivial power. However, we argue that setting $k=1$, our method cannot do better than random guessing.  Specifically, we show that for any vector $u\in \reals^p$, we $\de(\theta_0, \Omega_0, u) = 0$. By symmetry, assume that $|u_1|\le |u_2|$. Note that the point $z_0 = \pm(a \sign(u_1), 3a \sign(u_2), \dotsc, 3a\sign(u_p))\in \Omega_1\subset \Omega_0$. Further, $u^\sT z_0 = \pm(a|u_1| + 3a |u_2| + \dotsc + 3a|u_p|)$. By convexity of $\Omega_1$, we have that $\proj_u(\Omega_1)\supseteq A$ where $A = \{\alpha u: |\alpha| \le |u^\sT z_0|\}$.  In addition, we have $u^\sT \theta_ 0 = 2a (u_1+u_2)$ and using the assumption $|u_1|\le |u_2|$, we get $|u^\sT \theta_0| \le |u^\sT z_0|$. Therefore, $\proj_u(\theta_0) \in A \subseteq \proj_u(\Omega_1) \subset \proj_u(\Omega_0)$. This implies that $\de(\theta_0, \Omega_0, u) = 0$, meaning that we cannot do better than random guessing if the inference is done in the on-dimensional projected space. We refer to Figure~\ref{fig:schematic} for a schematic illustration of this example in $p=2$.
\begin{figure}[]
    \centering{
        \includegraphics[width = 8cm]{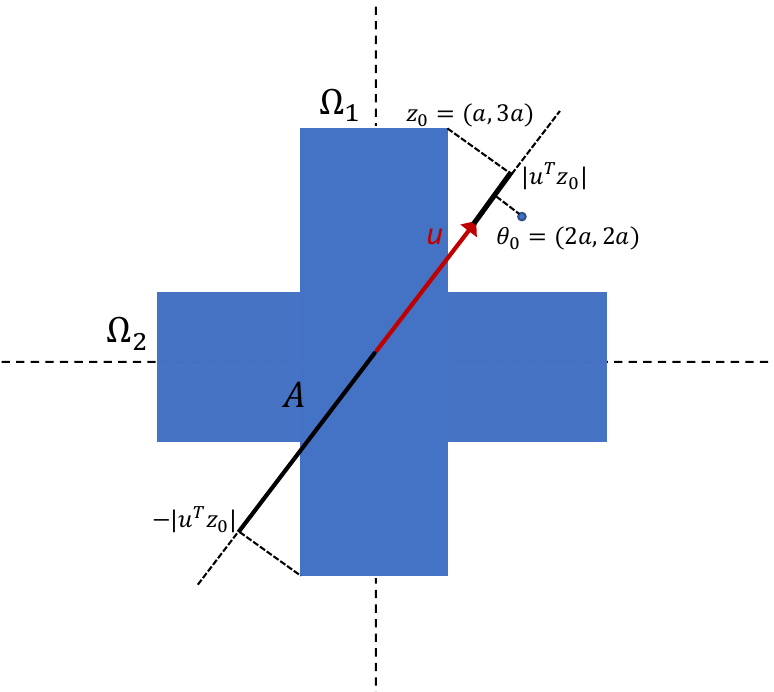}
        \caption{Illustration of the example of non-convex $\Omega_0$ discussed in Remark~\ref{rem:example-nonCvx} for $p=2$ }\label{fig:schematic}
        }
\end{figure}
\end{remark}
}

\rev{
\section{Approximate sparsity}\label{sec:app-sparsity}

With the aim of broadening the application of our proposed method, we relax the sparsity assumption of the model to a so-called approximate sparsity structure. 
Consider the linear model 
\begin{align}\label{eq:app-sparsity-model}
y = X\theta_*+w\,,
\end{align}
with $w\sim \normal(0,\sigma^2 \id_{n\times n})$, and $\theta_*\in \reals^p$ the unknown model parameters that is not necessarily sparse. 
However, we assume that there exists at least one sparse linear combination of the covariates that gets close to the true signal. This is formally stated as the approximate sparsity stated below, which is similar to the one introduced by
\cite{belloni2012sparse}.
\medskip

\begin{assumption}\label{ass:app-sparsity}
{\bf (Approximately Sparse Model).} The signal $X\theta_*$ is well approximated by a linear combination of  unknown $s_0\ge 1$ covariates:
\begin{align}
X\theta_* = X\theta_0 + r\,, \quad \|r\| =o_P(1)\,.
\end{align}
\end{assumption}
The approximate sparsity assumption in~\cite{belloni2012sparse} is weaker than the one we are imposing here, as the former allows for $\|r\| = O_P(\sqrt{s_0})$.

The next assumption is also introduced by~\cite{belloni2012sparse}, under the name of ``RF condition". This is basically an assumption on the moments of covariates and the noise component. In stating that we borrow the following empirical process notation from~\cite{belloni2012sparse}: 
$\E_n[f]\equiv \E_n[f(z_i)] \equiv \sum_{i=1}^n f(z_i)/n$ and $\bE[f] \equiv \E \E_n[f]  = \E \E_n [f(z_i)] = \sum_{i=1}^n \E[f(z_i)]/n$.

\begin{assumption}\label{ass:moment}
{\bf (Moment Condition).} Suppose that the following moment conditions holds:
\begin{itemize}
 \item[$(i)$] For a constant $C_2>0$, $\bE[y_i^2] + \bE[X_{ij}^2 y_i^2] + 1/\bE[X_{ij}^2 w_i^2] \le C_2$.
 \item[$(ii)$] We have $\max_{j\in [p]} \bE[|X_{ij}^3 w_i^3|]\le o(\sqrt{n/(\log p)^3})$, and also $s_0\log p = o(n)$.
 \item [$(iii)$] $\underset{i\in [n], j\in [p]}{\max}  X_{ij}^2 (s_0 \log p)/n \to 0$, in probability and $\underset{j\in [p]}{\max} |(\E_n - \bE) [X_{ij}^2w_i^2]| + |(\E_n - \bE) [X_{ij}^2y_i^2]| \to 0$, in probability. 
 \end{itemize}
\end{assumption}

The above moment condition was proposed in~\cite{belloni2012sparse} where they bound the estimate error of selection methods such as Lasso under approximate sparsity condition. Our lemma below provides a set of alternative conditions that, for sub-gaussian designs, imply the Moment condition~\ref{ass:moment}.
\begin{lemma}\label{lem:moment}
Suppose that the design $X$ has independent sub-gaussian centered rows with uniformly bounded sub-gaussian norm ($\|x_i\|_{\psi_2}\le C$). Assume that $y_i$ and $w_i$ have uniformly bounded conditional moments of order $4$, that is $\E(y_i^4|x_i) \le C'$ and $\E(w_i^4|x_i) \le C''$, for $i\in [n]$.  In addition, suppose that $s_0 = o(n/\log^2(p))$ and $\log p = o(n^{1/3})$. Then the Moment Condition~\ref{ass:moment} holds.
\end{lemma}
We refer to Appendix~\ref{proof:lem-moment} for the proof of Lemma~\ref{lem:moment}.

\medskip

\noindent{\bf Iterated Lasso.} Following~\cite{belloni2012sparse}, we consider a weighed Lasso estimator of $\theta_0$. Formally, let $\hth$ be given by
\begin{align}\label{eq:pen}
\hth = \arg\min_{\theta\in \reals^p} \Big\{\frac{1}{n} \|y - X\th\|^2 + \lambda \sum_{i=1}^p |\gamma_i \theta_i| \Big\}\,,
\end{align}
where the regularization $\lambda$ is chosen as
\begin{align}\label{lambda:iterated}
\lambda = \frac{2.2}{\sqrt{n}} \Phi^{-1} (1-0.1/(2p\log p))\,. 
\end{align}
The weights $\gamma_i$, $j\in [p]$ are ideally chosen as $\gamma_j = \sqrt{\E_n[X_{ij}^2 w_i^2]}$. But since the noise terms $w_i$ are unobserved this ideal option is not realizable. Hence, we use an iterative method proposed in~\cite{belloni2012sparse,belloni2014inference} to set the weights $\gamma_i$. (The resulting Lasso estimator $\hth$ is referred to as `iterated Lasso' in~\cite{belloni2012sparse,belloni2014inference}.) 
 The details of the procedure is described in Algorithm~\ref{alg:iterated-Lasso}. 

\begin{algorithm}[h]
\caption{\rev{Choosing weights in the iterated Lasso estimator}} \label{alg:iterated-Lasso}
\begin{algorithmic}[1]
\rev{
\REQUIRE response vector $y$, design matrix $X$, regularization parameter $\lambda$, number of iteration $K$.

\ENSURE estimator $\hth$

\STATE {\bf (initialization)} set $\gamma_j = \sqrt{\E_n[X_{ij}^2y_i^2]}$, for $j\in[p]$.

\FOR{$k = 1, 2, \dotsc, K$} 

\STATE compute $\hth$ estimator given by~\eqref{eq:pen}. 
\STATE update the weights as $\gamma_j = \sqrt{\E_n[X_{ij}^2 (y_i - x_i^\sT\hth)^2]}$.
\ENDFOR
}
\end{algorithmic}

\end{algorithm}

Our next theorem is analogous to Theorem~\ref{thm:typeI-random} and shows our procedure controls the type-I error for random designs under approximately sparse models.

\begin{thm}\label{thm:typeI-random-aps} 
Let $\Sigma\in \reals^{p\times p}$ such that $\sigma_{\min}(\Sigma) \ge C_{\min}>0$ and $\sigma_{\max}(\Sigma) \le C_{\max} <\infty$ and $\max_{i\in[p]} \Sigma_{ii}\le 1$. 
Suppose that the regression model~\eqref{eq:app-sparsity-model} is approximately sparse (Assumption~\ref{ass:app-sparsity}), and assume that the responses $y_i$ have uniformly bounded conditional moment of order 4, that is $\E(y_i^4|x_i)\le C'$ for $i\in [n]$ and a constant $C'>0$ independent of $n$.

Let $\hth$ be the iterated Lasso estimator using data $(y,X)$, given by~\eqref{eq:pen}. Consider an arbitrary $U\in \reals^{p\times k}$, with $U^\sT U = I$, that is independent of the samples $\{(x_i,y_i)\}_{i=1}^n$. Construct a debiased estimator $\dgamma$ as in~\eqref{dgamma} with 
$\mu =a\sqrt{(\log p)/n}$, and $a^2> 48e^2\kappa^4C_{\max}/C_{\min}$. 
In addition, suppose that $\lim \sup_{n\to \infty}  \mu  (\max_{i\in [k]}\, \|u_i\|_1)\le c'$, for some constant $0<c'<1$,  $s_0 = o(\sqrt{n}/{\log p})$ and $\log p = o(n^{1/3})$.

For the test $R_X$ defined in Equation~\eqref{detection}, and for any $\alpha \in [0,1]$, we have
\begin{align}
\underset{n\to \infty}{\lim\sup}\;\; \alpha_n(R_X) \le \alpha\,.
\end{align}
\end{thm}


We refer to Section~\ref{proof:thm-typeI-random-aps} for the proof of Theorem~\ref{thm:typeI-random-aps}.
}

\rev{
\section{Extension to Non-Gaussian Noise}\label{sec:non-gaussian}
Our analysis can be extended to the case of non-gaussian noise measurements. Specifically, suppose that the noise term $w_i$ satisfies 
\begin{align}\label{eq:noise-relax}
\E(w_i|X) = 0,\quad \E(w_i^2|X) = \sigma^2, \quad \E(|w_i|^{4+a} |X) \le B\,,
\end{align}
for some constants $a, B>0$,  and $1\le i\le n$.

Recall that our analysis is based on a bias-variance decomposition of the estimate $\dgamma$ as in Lemma~\ref{propo:bias-size2}. The bias term $\|\Delta\|_\infty$ can be bounded as
\[
\|\Delta\|_\infty \le \sqrt{n} \|G^\sT\hSigma - U\|_\infty \|\theta_0 - \hth\|_1\,.
\]
The first term does not involve the noise term $w$ and can be treated as before. For bounding $\|\theta_0-\hth\|_1$, we used the result of \cite[Theorem 1]{belloni2012sparse} (See Proposition~\ref{propo:aps} in the Appendix) that also applies to non-gaussian noise as long as the moment conditions (Assumption~\ref{ass:moment}) hold, which by Lemma~\ref{lem:moment}, for sub-gaussian designs it reduces to requiring the noise variables $w_i$ have bounded conditional moment of order 4.  

So the remaining part is characterizing the limiting distribution of $Z$. To this end, we will show that the Lindeberg condition holds and hence $Z$ admits an asymptotically normal distribution by virtue of central limit theorem. 

Similar to the approach taken in~\cite{javanmard2014confidence}, we slightly modify our construction of the decorrelating matrix $G$ to ensure the Lindeberg condition holds. Let $\G = [g_1|\dotsc|g_k]\in \reals^{p\times k}$, where each $g_i$ is obtained by solving the following optimization problems for each $1\le i\le k$:

\begin{align}\label{OPT:M-n-n}
\begin{split}
&\text{minimize} \quad g^\sT \hSigma g\\
&\text{subject to} \;\; \|\hSigma g - u_i\|_\infty \le \mu\\
&\quad \quad\quad \quad \quad \|Xg\|_\infty\le n^\beta\,, \quad \text{ for arbitrary fixed } 0<\beta<1/2\,.
\end{split}
\end{align}

Our following proposition shows that $Z$ admits an asymptotically normal distribution in the non-gaussian setting.
\begin{propo}\label{lindberg}
Suppose that the noise variables $w_i$ are independent with $\E(w_i|X) = 0$, $\E(w_i^2|X) = \sigma^2$ and $\E(|w_i|^{4+a}|X) \le B$ for some $a>4\beta/(1-2\beta)$. Let $G = [g_1|\dotsc|g_k] \in \reals^{p\times k}$ be the matrix constructed by solving optimization problem \eqref{OPT:M-n-n}. For $i\in [p]$, define
\begin{align}
Z_i = \frac{1}{\sqrt{n}} \frac{g_i^\sT X^\sT w}{ \sigma (g_i^\sT \hSigma g_i)^{1/2}}\,.
\end{align}
Suppose that the assumptions of Theorem~\ref{thm:typeI-random-aps} hold. 
Then,  for any sequence $i = i(n)\in [p]$, and any $x\in \reals$, we have
\[
\lim_{n\to\infty} \prob(Z_i\le x|X) = \Phi(x)\,,
\]
with $\Phi(x)$ indicating the cdf of standard normal variable.
\end{propo}

We refer to Appendix~\ref{app:lindberg} for the proof of Proposition~\ref{lindberg}.

}

\section{Discussion}\label{sec:Other}
It is useful to study the proposed methodology for some specific choices of $\Omega_0$ and discuss its optimality.

{\bf Example 1 (Predictions).} Fix an arbitrary $c \in \reals$ and consider the set $\Omega_0  = \{\theta: \xi^\sT \theta  = c\}$. This corresponds to the set where the (noiseless) unobserved response on the new feature vector $\xi$ is $c$. We can use our methodology to test $H_0: \theta_0\in \Omega_0$ versus its alternative. Further, by duality of hypothesis testing and confidence intervals, our methodology provides confidence intervals for a linear functional of the form $\xi^\sT \theta_0$.  

%

Computing $u$ from~\eqref{eq:u-opt} in this case gives $u = \xi/\|\xi\|$. Since $\xi$ is independent of $(y,X)$, the data splitting step in the procedure becomes superfluous. 
By duality, we construct   $(1-\alpha)$ confidence interval for $\xi^\sT \th_0$ by finding the range of values $c$ such that the rule fails to reject $H_0$ at level $\alpha$. This is formalized in the next lemma.
\begin{lemma}\label{lem:prediction}
Consider a sequence of design matrices $X\in \reals^{n\times p}$, with dimensions $n,p \to \infty$, $p = p(n)\to \infty$ satisfying the assumptions of Theorem~\ref{thm:typeI}.
For given  $\alpha\in (0,1)$, define $C(\alpha) = [c_{\min},c_{\max}]$ 
with 
\begin{align}\label{eq:CI1}
c_{\min} &= \| \xi\| \dgamma -\frac{\hsigma}{\sqrt{n}} \sqrt{g^\sT\hSigma g}\, z_{\alpha/2}\|\xi\|_2  \,,\\
c_{\max}&= \| \xi\| \dgamma + \frac{\hsigma}{\sqrt{n}} \sqrt{g^\sT\hSigma g}\,z_{\alpha/2}\|\xi\|_2\,, 
\end{align}
where $\dgamma$ is the debiased estimator given by~\eqref{dgamma-simple}  with $u = \xi/\|\xi\|$. Then,
\begin{align}
\underset{n\to \infty}{\liminf}\; \prob\left(\xi^\sT \th_0 \in C(\alpha) \right) \ge 1-\alpha\,.
\end{align}
\end{lemma}
We refer to Appendix~\ref{proof:lem-prediction} for the proof of Lemma~\ref{lem:prediction}. The constructed confidence interval has length of rate $\|\xi\|/\sqrt{n}$. In~\cite{cai2015confidence}, it is shown that the minimax expected length of confidence intervals for $\xi^\sT\theta_0$, with a sparse vector $\xi$ (i.e., $\|\xi\|_0 = O(s_0)$) is $\|\xi\|(1/\sqrt{n} + s_0(\log p) /n)$. Therefore, in the regime $s_0 = o(\sqrt{n}/\log p)$, which is the focus of the current paper, the constructed confidence intervals are minimax rate optimal. It is worth noting that the confidence interval defined 
in Lemma~\ref{lem:prediction} is similar to the one proposed by~\cite{cai2015confidence}.  For the case of non-sparse $\xi$, \cite{cai2015confidence} establishes the minimax rate $\|\xi\|_\infty s_0\sqrt{(\log p)/n}$ for the expected length of confidence interval for $\xi^\sT \theta_0$, and hence our construction~\eqref{eq:CI1} has an optimality gap in this case.
\bigskip

{\bf Example 2 (Quadratic forms).} As another example we apply our framework to testing squared-$\ell_2$ norm of $\theta_0$. Consider the set $\Omega_0(c)  = \{\theta: \|\theta\|_2^2  = c\}$, where $c\ge 0$ is a fixed arbitrary constant.
We use the proposed framework to test the null hypothesis $H_0: \theta_0\in \Omega_0(c)$. Computing $u$ from~\eqref{eq:u-opt} in this case gives $u = \hth^{(1)}/\|\hth^{(1)}\|$.
We next use the duality between hypothesis testing and confidence intervals to construct confidence intervals for $\|\theta_0\|_2^2$.
\rev{
\begin{lemma}\label{lem:quadratic-CI}
Consider a sequence of design matrices $X\in \reals^{n\times p}$, with dimensions $n,p \to \infty$, $p = p(n)\to \infty$ satisfying the assumptions of Theorem~\ref{thm:typeI-random}).
For given  $\alpha\in (0,1)$, define $C(\alpha) = [c_{\min},c_{\max}]$ with
\begin{align}
c_{\min} &= \left( 2\dgamma \|\hth^{(1)}\| - \|\hth^{(1)}\|^2 - L \right)_+\,, \quad 
c_{\max}= \left( 2\dgamma \|\hth^{(1)}\| - \|\hth^{(1)}\|^2 + L\right)\,,\label{eq:CI2}\\
 L&= \|\hth^{(1)}\|  \sqrt{g^\sT\hSigma g}\,(1+o(1))  \,  \frac{\hsigma z_{\alpha/2}}{\sqrt{n}}   \,, \label{L}
\end{align}
where $a_+ = \max(a,0)$ and $\dgamma$ is the debiased estimator given by~\eqref{dgamma-simple} with $u = \hth^{(1)}/ \|\hth^{(1)}\|$. Then,
\begin{align}
\underset{n\to \infty}{\liminf}\; \prob\left(\|\th_0\|_2^2 \in C(\alpha) \right) \ge 1-\alpha\,.
\end{align}
\end{lemma}}
We give the proof of Lemma~\ref{lem:quadratic-CI} in Appendix~\ref{proof:lem-quadratic-CI}.
\bigskip

{\bf Example 3 (Testing $\theta_{\min}$ condition).}
For a given $c>0$, define the set $\Omega_{0} = \{\th\in \reals^p:\; \min_{j\in \supp(\th)} |\th_j| \ge c\} $.  Apart from the importance of this example as discussed in the introduction, it differs from previous example in that
the set $\Omega_{0}$ is non-convex and disconnected. Recall that the guideline~\eqref{eq:u-opt} was provided for convex sets $\Omega_0$, which is not true in this example. 

Before proposing a choice of $U$ for this example, we state a lemma.
\begin{lemma}\label{proj-betamin}
Let $v\in \reals^p$ and define $\th\in \reals^p$ with $\th_i = \cS(v_i,c)$, where 
\begin{align}
\cS(x,c) = \begin{cases}
x & |x|\ge c\,,\\
c & x\in (c/2,c)\\
0 & x\in [-c/2,c/2]\\
-c& x\in (-c,-c/2)
\end{cases} 
\end{align}

Then $\th$ is a solution to $\min_{\th\in \reals^p} \|D(v -\th)\|_\infty$, subject to $\th\in \Omega_0$, for any diagonal matrix $D$.  
\end{lemma}
Proof of Lemma~\ref{proj-betamin} is straightforward and is omitted.

In the numerical experiments, we apply our framework for this example with $k = 1$ and $U  = u\in \reals^{p}$ given by:
\begin{align}\label{eq:u-beta-min}
u = e_{i^\star}\,,\quad i^\star \equiv \arg\max_{i\in [p]}\, \Big|\hth^{(1)}_i - \cS(\hth^{(1)}_i,c))\Big|\,.
\end{align}

 We refer to Appendix~\ref{app:justification} for a justification for this choice. 
%
By using Lemma~\ref{proj-betamin}, the test statistic in this case amounts to $T_n = |d(\dgamma-\cS(\dgamma,c))|$ (See step 5 of the algorithm presented in Section~\ref{sec:choice-U}).

\rev{
\subsection{Prior art}\label{sec:prior-art}
The inference problem~\eqref{eq:H0-HA} studied in this paper is very general and encompasses several important problems such as the examples discussed in Section~\ref{sec:motivation}.
For specific choices of set $\Omega_0$, one may use the structure of the set $\Omega_0$ to come up with methods with higher statistical power. However, in the sequel we argue that for three classes of inferential problems, our proposed framework either recovers the previously proposed methods for that specific problem, or have comparable performance. We also contrast the underlying assumptions of our framework and those of other methods designed for these specialized problems. 
\medskip

{\bf 1. Inference on prediction:}
As discussed in Section~\ref{sec:Other}, for inference on linear functions $\gamma_0 =\xi^\sT \theta_0$ (predictions), our framework proposes $u = \xi/\|\xi\|$ and construct a debiased estimator of $\gamma_0$ taking the following form
\begin{align}\label{eq:debias-linear}
\dgamma = \frac{\xi^\sT}{\|\xi\|} \htheta + \frac{1}{n} g^\sT X^\sT (y - X\htheta)\,,
\end{align}
with $g$ is obtained by solving optimization~\eqref{OPT:M-n}. As argued for the case of random designs with population covariance $\Sigma$, this implies $g \approx \Sigma^{-1} \xi/\|\xi\|$. 
As also discussed earlier in the introduction and previous section, a similar approach has been used by~\cite{cai2015confidence} and they prove that the resulting confidence interval would be minimax rate optimal. It is indeed an appealing property of our method that, despite its generality,  it recovers the method of~\cite{cai2015confidence} for this specific case and enjoys minimax optimality. 

\smallskip

$\bullet$ {\bf Assumptions:} In terms of assumptions,~\cite{cai2015confidence} focuses on high-dimensional linear models with gaussian designs (rows of design matrix are drawn i.i.d from a multivariate normal distribution),
sparse parameter vector and gaussian measurement noise. Our analysis in Section~\ref{sec:Results} considers sub-gaussian random designs (Theorem~\ref{thm:typeI-random}) and coherent fixed design (Theorem~\ref{thm:typeI}). We also extended our analysis to \emph{approximately} sparse models (Section \ref{sec:app-sparsity}) and non-gaussian noise (Section~\ref{sec:non-gaussian}).

\smallskip

$\bullet$ {\bf Least-favorable one-dimensional sub-model:} 
It is worth noting that the form of debiasing~\eqref{eq:debias-linear} for linear functionals of $\theta$ can also be derived from the perspective of least-favorable scores discussed in an earlier work~\cite{zhang2014confidence}. Akin to the semi-parametric models, consider the one-dimensional sub-model $\{\theta_0 + u\phi,\, |\phi| < \eps_*\}$ with $\eps_*\to 0$, $\phi$ scalar and $u\in \reals^p$. By imposing the constraint $\xi^\sT u =1$, we have $\xi^\sT(\theta_0+u\phi) - \xi^\sT \theta_0 = \phi$. The idea of~\cite{zhang2014confidence} is to look for the least favorable submodels at $\theta_0$, given by $\theta_0 + u\phi$ with $u_0$ the direction that minimizes Fishers information. For the log-likelihood $\ell_i(\theta_0) = \ell(\theta_0|y_i, x_i)$, recall that the Fisher information operator at $\theta$ is defined as $F = -\E(\ddot{\ell_i}(\theta))$ and  for linear regression with gaussian errors, we have $F = \tfrac1\sigma^2 \E(x_ix_i^\sT) = \tfrac1\sigma^2 \Sigma$. The least-favorable direction in the sub-model is then given by
\[
u_0 = \arg\min_u \{u^\sT \Sigma u:\, \xi^\sT u = 1\} = \Sigma^{-1}\xi/(\xi^\sT\Sigma^{-1}\xi)\,.
\]
Following~\cite{zhang2014confidence}, one can construct a low-dimensional projection estimator (LDPE) as a one-step maximum likelihood correction of $\htheta$ in the direction of the least favorable sub-model $u$ as follows
\begin{align}
\dgamma &= {\xi^\sT} \htheta + \arg\max_{\phi\in\reals} \sum_{i=1}^n \ell_i(\htheta+u\phi)\nonumber\\
 &=  {\xi^\sT} \htheta + \frac{u^\sT X^\sT (y- X\htheta)}{\|Xu\|^2} = {\xi^\sT} \htheta + \frac{\xi^\sT\Sigma^{-1}\xi}{\|X\Sigma^{-1}\xi\|^2}\, \xi^\sT\Sigma^{-1} X^\sT (y- X\htheta)\nonumber\\
 &\approx {\xi^\sT} \htheta + \frac{1}{n} \xi^\sT\Sigma^{-1} X^\sT (y- X\htheta)\,,\label{dgamma-score}
\end{align}      
where in the last step we replaced the denominator by its expectation. Comparing \eqref{dgamma-score} with \eqref{eq:debias-linear} we see that (up to a normalization by $\|\xi\|$) they are the same if $g = \Sigma^{-1}\xi$. However, $\Sigma$ is unknown in general and optimization~\eqref{OPT:M-n} try to find $g\approx \Sigma^{-1}\xi$ that also minimizes the variance of the obtained debiased estimator.

\smallskip

$\bullet$ {\bf Choice of $k$ and effect of sample splitting:} 
Our procedure uses sample splitting to find the best subspace $U$ for the sake of statistical power. On one side, the sample splitting incurs loss in power as we are using only half of data points. On the other side, the purpose of sample splitting was to choose $U$ so as to increase the power. To understand this trade-off we consider the following inference problem. 
Consider a function $h:\reals^p\mapsto \reals^q$ defined as $h(\theta) = (\xi_1^\sT \theta, \dotsc, \xi_q^\sT \theta)$, for a linearly independent set $\{\xi_1,\dotsc, \xi_q\}$. The goal is to do inference on the value of $h(\theta_0)$. 
We consider the following two methods of choosing $U$ in constructing the debiased estimator:

\begin{enumerate}
\item \emph{Method 1:} We let $k = q$ and $U$ be a basis for the space spanned by $\{\xi_1,\dotsc, \xi_q\}$. This method does not require any sample splitting.
\item \emph{Method 2:} Define $\Omega_0 = \{\theta: h(\theta) = c\}$, for a given $c>0$. Since $\Omega_0(c)$ is convex, our methodology sets $k=1$ and chooses $u$ as in~\eqref{eq:u-opt}. Here we require sample splitting for $q\ge2$. (cf. Section \ref{sec:convex})
\end{enumerate}
Note that the two methods become identical for $q = 1$.
We next compare (the analytical lower bound on) the statistical power of these two methods for choosing $U$. Let $\eta_u = \de(\htheta, \Omega_0; u)$ and $\eta_u = \de(\htheta, \Omega_0; U)$, with $u$ given by \eqref{eq:u-opt}
and $U$ a basis for the space $\{\xi_1, \dotsc, \xi_q\}$. Using Theorem~\ref{thm:power} and Equation~\eqref{power-LB1}, the lower bound for the power of method 1 and method 2 are respectively given by  $F(\alpha,\tfrac{1}{\hsigma}\sqrt{nC_{\min}} \eta_U, q)$ and $F(\alpha, \frac{1}{\sqrt{2}\hsigma}\sqrt{nC_{\min}} \eta_u, 1)$. 
Furthermore, by Equation~\eqref{eq:dum4} we have $\eta_u \ge \eta_U$ and since $F(\alpha, x,k)$ is increasing in $x$, we get $F(\alpha,  \frac{1}{\sqrt{2}\hsigma}\sqrt{nC_{\min}} \eta_u, 1) \ge F(\alpha,  \frac{1}{\sqrt{2}\hsigma}\sqrt{nC_{\min}} \eta_U, 1)$. In summary, we have
\begin{align}
\lim\inf_{n\to\infty} \frac{{\sf power}_1(n)}{ F\Big(\alpha,\tfrac{1}{\hsigma}\sqrt{nC_{\min}} \eta_U, q\Big)} \ge 1\,, \quad \quad  \lim\inf_{n\to\infty} \frac{{\sf power}_2(n)}{ F\Big(\alpha,  \frac{1}{\sqrt{2}\hsigma}\sqrt{nC_{\min}} \eta_U, 1\Big)} \ge 1\,.
\end{align}
The above lower bounds nicely capture the tradeoff between the choice of $k$ and the sample splitting. The function $F(\alpha,x,k)$ is decreasing in $k$ which supports the use of $k=1$, but the function is increasing in the $x$ and hence decreases under sample splitting. To understand this tradeoff we basically need to compare $F(\alpha, x, 1)$ and $F(\alpha, \sqrt{2}x, q)$, with $x = \frac{1}{\sqrt{2}\hsigma}\sqrt{nC_{\min}} \eta_U$.
In Figure~\ref{fig:F}, we have plotted these curves for $\alpha = 0.05$ and several values of $q$. As we see for small values of signal strength $x$, method 2 ($k=1$ and sample splitting) outperforms, while for larger signal strength $x$, method 1 ($k>1$ and no sample splitting) prevails.

\begin{figure}[]
    \centering{
        \includegraphics[width = 7cm]{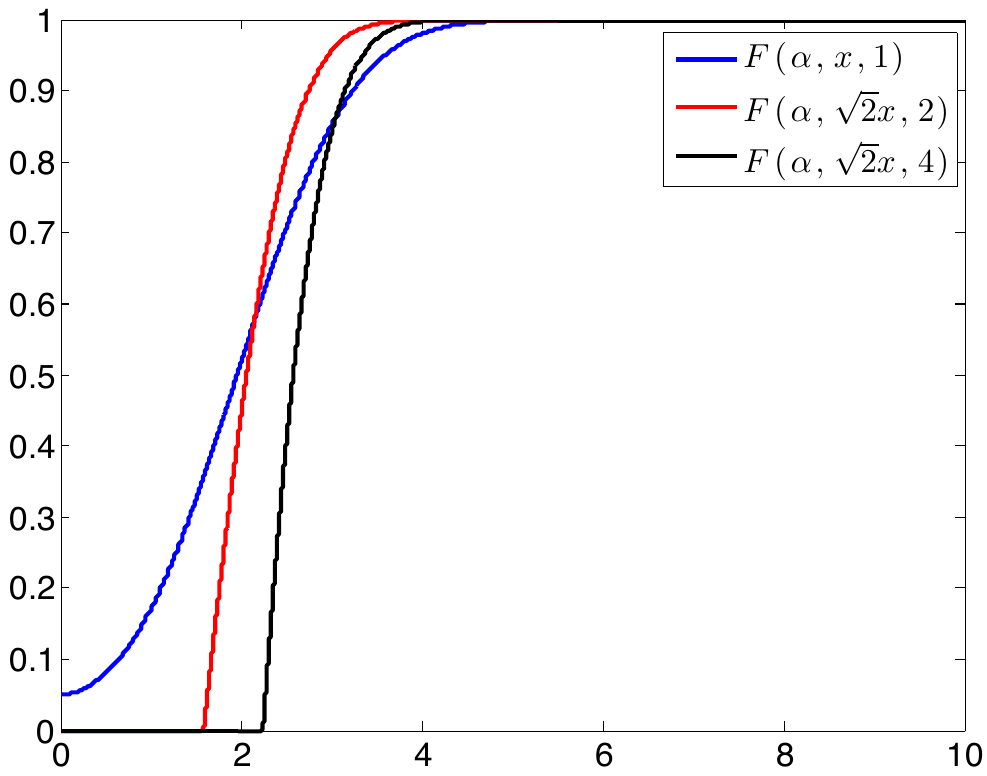}
        \put(-95,-7){$x$}
       \put(-100,-15){\phantom{AA}}
        \caption{Plot of $F(\alpha,x,1)$ and $F(\alpha,\sqrt{2}x, q)$ for $q = 2, 4$ and $\alpha = 0.05$}\label{fig:F}
        }
\end{figure}

\medskip

{\bf 2. Inference on quadratic forms of parameters:}
The work~\cite{janson2016eigenprism} proposed \emph{EigenPrism}, a procedure to construct two-sided confidence interval for the signal squared magnitude $\|\theta_0\|^2$. An appealing property of this procedure is that, albeit its applicability to the high-dimensional setting ($p>n$), it does not make any assumption on the coefficient sparsity. However, it is theoretically justified only for standard gaussian designs where $X_{ij}\sim \normal(0,1)$, independently. As explained in \cite{janson2016eigenprism}, this assumption is crucial because it ensures that in the SVD of $X = UD V^\sT$, the columns of $V$ are uniformly distributed on the unit sphere, and hence allows for computing the expectation and variance of inner products of columns of $V$ with $\theta_0$, which constitutes a main building component of EigenPrism. By contrast, our procedure (when specialized to inference on quadratic forms of parameters as discussed in Section~\ref{sec:Other}, Example 2) applies to a much broader family of sub-gaussian random designs, but assumes the coefficient sparsity  $s_0 = o(\sqrt{n}/\log p)$. 

In the limit $n, p\to \infty$ and $n/p \to \gamma\in (0,1)$, the length of confidence intervals constructed by EigenPrism for $\|\theta_0\|^2$ works out at $C_\gamma  (\|\theta_0\|^2+ \sigma^2) \frac{z_{\alpha/2}}{\sqrt{n}}$, with $C_\gamma$ a numerical constant defined based on Marcenko-Pastur distribution with parameter $\gamma$. By comparison, using Lemma~\ref{lem:quadratic-CI}, the confidence interval obtained by our method is of length $2L < \frac{2z_{\alpha/2}}{\sqrt{C_{\min}}} \|\htheta^{(1)}\| \tfrac{\sigma}{\sqrt{n}}$.  As we see the length of confidence intervals for $\|\theta_0\|^2$ from both methods scale at rate $1/\sqrt{n}$.
 
 \medskip
 
 {\bf 3. Inference on individual parameters:} As discussed in Section~\ref{sec:motivation}, for the special case of inference on an individual model parameter, our approach recovers the debiasing method of~\cite{javanmard2014confidence}. Similar debiasing approach (with different construction of the the decorrelating matrix, using node-wise regression) was proposed in~\cite{zhang2014confidence,van2014asymptotically} and its validity is proved under the assumption that the precision matrix $\Sigma^{-1}$ is sparse. 
The work~\cite{belloni2014inference} has proposed a significantly different approach for doing inference on an an individual parameter, called ``post-double selection". Suppose that we are interested in parameter $\theta_i$. This method consists of two selection steps: 1) Let $I_1$ be the covariates selected by Lasso in regressing columns $i$ of the design matrix on the other columns; 2) Let $I_2$  denote the covariates selected by Lasso in regressing $y$ on the design $X$. The estimation of parameter $\theta_i$ is then defined as the least squares estimator obtained by regression $y$ on $x_i$ and the selected features $I_1\cup I_2$ (we may expand this set to also include other features that the statistician thinks are relevant). It is shown that the post-double estimator obeys an asymptotically normal distribution.

The limiting distribution of the post-double estimator is characterized under approximate sparsity structure and also applies to non-gaussian noise as well, as far as some moment conditions (similar to Assumption~\ref{ass:moment}) hold. Let us stress that the approximate sparsity assumption in~\cite{belloni2014inference} is much weaker than ours in that it allows for $\|r\| = O_P(\sqrt{s_0})$, while we require $\|r\| = o_P(1)$. In addition, the analysis of the post-double estimator  extends to possibly heteroscedastic noise distributions. 
}

\section{Numerical illustration}\label{sec:numerical}

In this section, we examine the performance of our inference framework in terms of coverage rate and length of confidence intervals, type I error and statistical power under different setups.
We consider linear model~\eqref{eq:NoisyModel} where the design matrix $X\in \reals^{n\times p}$ has i.i.d rows generated from $\normal(0,\Sigma)$, with $\Sigma\in \reals^{p\times p}$ being the toeplitz matrix 
$\Sigma_{i,j} = \rho^{|i-j|}$.  For coefficient parameter $\th_0$, we consider a uniformly random support  (set of nonzero parameters) $S\subseteq [p]$, with $|S| = s_0$. 
The measurement errors are $w_i \sim \normal(0,1)$.

\subsection{Testing $\theta_{\min}$ condition}
We consider the set $\Omega_0 = \{\th: \min_{j\in \supp(\th_0)} |\th_{0,j}| \ge c\}$ and the null hypothesis $H_0: \, \theta_0\in \Omega_0$. As explained in Section~\ref{sec:Other} (Example 3), the set $\Omega_0$ is non-convex (indeed disconnected) and \rev{we consider one-dimensional projection of the problems along the direction $u$ given by~\eqref{eq:u-beta-min} for this example}. For the scaled Lasso estimator $\Lth$, given by~\eqref{scaledLasso}, we set the regularization parameter $\lambda = \sqrt{2.05 (\log  p)/n}$.
Further, the parameter $\mu$ in constructing the debiased estimator (see optimization problem~\eqref{OPT:M-n}) is set to  $\mu = 2\sqrt{(\log p)/n}$. We set $p = 1000$, $n = 600$, $s_0 =10$.
 \rev{The nonzero parameters $\theta_{0,i}$, $i\in S$, are chosen as $0.1, 0.2, \dotsc, 1$.} We set $\alpha = 0.05$ and vary the values of $c$ and $\rho$. The rejection probabilities are computed based on 300 random samples for each value of pair $(c,\rho)$. 
\rev{When $c\le 0.1$, $H_0$ holds and thus the rejection probability corresponds to the type I error. When $c>0.1$, the rejection probability corresponds to the power of the test. The results are reported in Table~\ref{tbl:beta-min}. As we see in Table~\ref{tbl:beta-min}(a), type I error is controlled below the desired level $\alpha = 0.05$. Also, as evident in Table~\ref{tbl:beta-min}(b), the power of our test increases at a very fast rate as $c$ increases.}

\begin{table}[h]%
\rev{
  \centering
  \subtable[Type I error ($\%$)]{
  \begin{tabular}{l|cccc}
    \hline
    $c\backslash \rho$ & $0.2$ & $0.4$ & $0.6$ & $0.8$ \\
    \hline
    $0.02$ &0.00 &0.004&1.33& 2.33 \\
    $0.04$ &0.33 &1.66& 2.33& 3.00 \\
    $0.06$ &1.66 &2.00&3.00& 3.66 \\
    $0.08$ &3.33 &4.33 &3.66& 4.66 \\
    $0.1$ & 3.00& 4.00 & 4.66 &4.33 \\
    
    \hline
  \end{tabular}
}
  \qquad\qquad
  \subtable[Statistical power ($\%$)]{
  \begin{tabular}{l|cccc}
    \hline
    $c\backslash \rho$ & $0.2$ & $0.4$ & $0.6$ & $0.8$ \\
    \hline
    $0.2$ &8.00 &10.66 &18.66 &14.33 \\
    $0.3$ &17.33 &24.66&28.66 & 35.33 \\
    $0.4$ &86.00 & 93.33 &92.66 &84.66 \\
    $0.5$ &90.00 &88.00&97.33&86.66 \\
     $0.6$ &100.00 &88.33&100.00&100.00 \\
    \hline
  \end{tabular}
}
  \caption{\rev{Type I error and statistical power for $H_0: \, \min_{j\in \supp(\th_0)} |\th_{0,j}| \ge c$, for significance level $\alpha = 0.05$.}}%
  \label{tbl:beta-min}%
  }
\end{table}

\subsection{Confidence intervals for linear functions} \label{CI-linear} We use our methodology to construct $95\%$ confidence intervals for functions of the form $\xi^\sT \th_0$. We set $p = 3000$, $s_0 = 30$
and choose the correlation parameter $\rho = 0.5$. \rev{The model parameters are set as follows. We set $\theta_{0,j} = 0.5$ for $j=1, \dotsc, s_0$, and $\theta_{0,j} = 0.5/(j-s_0+1)$, for $j=s_0+1, \dotsc, p$.} 

We construct confidence intervals according to Lemma~\eqref{lem:prediction}. We choose fives vectors $\xi_1, \xi_2,\dotsc, \xi_5$ as eigenvectors of $\Sigma$ with well-separated eigenvalues. Specifically, sorting the eigenvalues of $\Sigma$ as $\sigma_1 \ge \sigma_2 \ge \dotsc\ge \sigma_{3000}$, we choose the eigenvectors corresponding to $\sigma_1$, $\sigma_{750}$, $\sigma_{1500}$, $\sigma_{2250}$, $\sigma_{3000}$.
For each $\xi_i$, we vary $n$ in $\{1000, 1200, 1400, \dotsc, 2600\}$. For each configuration $(\xi_i,n)$, we consider $300$ independent realizations of measurement noise and on each realization, we construct 
$95\%$ confidence interval for $\xi_i^\sT \th_0$ based on Lemma~\eqref{lem:prediction}.

In Figure~\ref{fig:coverage}, we plot the average coverage probability of constructed confidence intervals for each configuration. Each curve corresponds to one of the vectors $\xi_i$. As we see, the coverage probability
for all of them and across different values of $n$ is close to the nominal value.

 In Figure~\ref{fig:CIL}, we plot the average length of confidence intervals as we vary the sample size $n$ in the log-log scale.
 As evident from the figure, the length of confidence intervals scales as $1/\sqrt{n}$.  

 %
\begin{figure}[]
    \centering
    \subfigure[\rev{Coverage of confidence intervals}]{
        \includegraphics[width = 6.5cm]{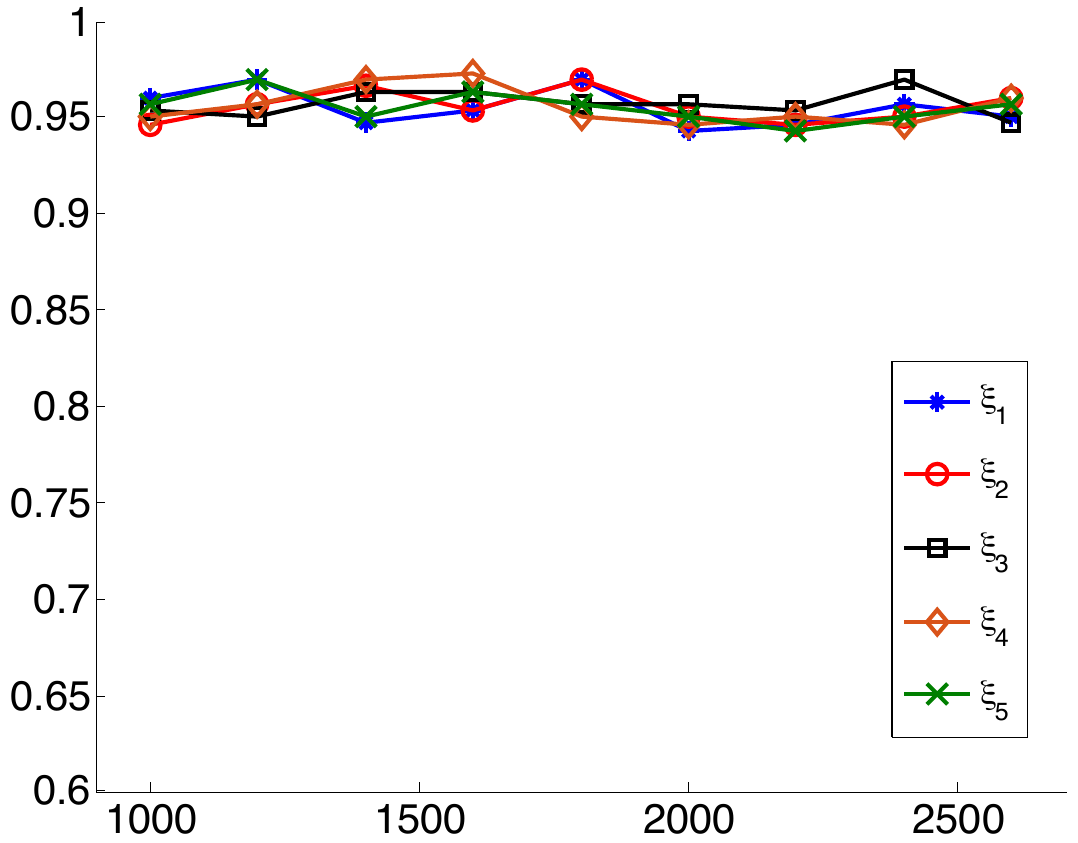}
        \put(-200,45){\rotatebox{90}{{\scriptsize $95\%$ CI Coverage}}}
        \put(-95,-7){{\scriptsize $n$}}
        \put(-100,-15){\phantom{AA}}
        \label{fig:coverage}
        }
        \hspace{1cm}
    \subfigure[\rev{Confidence interval widths}]{
        \includegraphics[width=6.5cm]{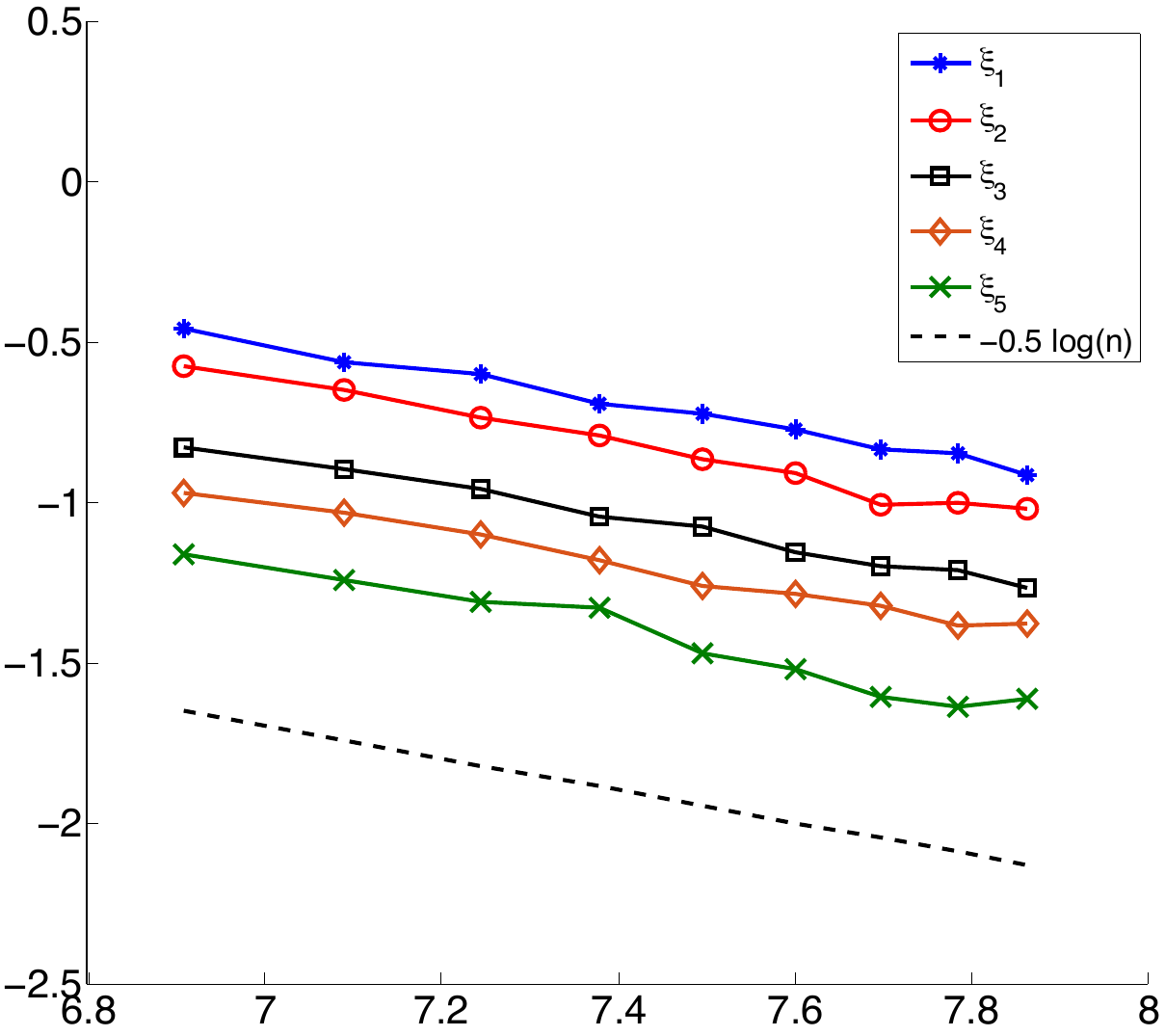}
         \put(-195,55){\rotatebox{90}{{\scriptsize log(CI width)}}}
         \put(-90,-7){{\scriptsize $\log(n)$}}
        \put(-100,-15){\phantom{AA}}
        \label{fig:CIL}
        }
    \caption{\rev{(a) Coverage of $95\%$ confidence intervals~\eqref{eq:CI1} for linear functions $\<\xi,\th_0\>$ versus sample size $n$. (b) Confidence interval widths versus sample size $n$. Here $p = 3000$, $s_0 = 30$, $\rho = 0.5$, and the model parameters are approximately sparse as described in Section~\ref{CI-linear}.}}\label{fig:CI1}
\end{figure} 
%

 %
%
\subsection{Testing for the non-negative cone}
\begin{table}[]
\rev{
  \centering
  \subtable[Type I error ($\%$)]{
  \begin{tabular}{l|cccc}
    \hline
    $b\backslash \rho$ & $0.2$ & $0.4$ & $0.6$ & $0.8$ \\
    \hline
    
    $1$ &2.00 &2.00&2.00&3.33 \\
    $0.8$ &0.66 &2.33&2.33&2.66 \\
    $0.6$ &3.00 &3.66&1.00&2.66 \\
    $0.4$ &2.66 &2.33&1.33&2.00 \\
    $0.2$ & 2.33&1.66&2.33&3.66 \\

    \hline
  \end{tabular}
}
  \qquad\qquad
  \subtable[Statistical power ($\%$)]{
  \begin{tabular}{l|cccc}
    \hline
    $b\backslash \rho$ & $0.2$ & $0.4$ & $0.6$ & $0.8$ \\
    \hline
    $-0.2$ &35.33 &68.00&78.00&80.00 \\
    $-0.4$ &99.33 & 100.00& 100.00&100.00  \\
    $-0.6$ &100.00 &100.00&100.00&100.00 \\
    $-0.8$&100.00 & 100.00 &100.00 & 100.00\\
    $-1$ &100.00&100.00&100.00&100.00\\
    \hline
  \end{tabular}
}
  \caption{\rev{Testing in the non-negative cone, $(n,s_0,p) =(600,10,1000)$. The non-zero entries have magnitude  $b$, and the covariance $\Sigma_{ij} =\rho^{|i-j|}$. } }%
  \label{tbl:nonneg-cone}%
  }
\end{table}
Define $\Omega_0  =\{\theta: \theta_i \ge 0 \text{ for all $i$}\}$ as the non-negative cone. In this section, we test whether $\theta_0 \in \Omega_0$ versus  $\theta_0 \notin \Omega_0$. The null model is generated as follows. \rev{ The nonzero entries in support $S$ are chosen as $b, b/2, b/3, \dotsc, b/s_0$, where $s_0 = |S|$ and $b>0$. The entries outside $S$ are set to zero.
The alternative model is generated similar where $b$ is replaced by $-b$.} 
As in the previous sections, the design matrix $X\in \reals^{n\times p}$ has i.i.d rows generated from $\normal(0,\Sigma)$, with $\Sigma\in \reals^{p\times p}$ being the toeplitz matrix 
$\Sigma_{i,j} = \rho^{|i-j|}$, and measurement errors $w_i \sim \normal(0,1)$, with parameters $(n,s_0,p) = (600,10,1000)$. We set $\alpha = 0.05$ and vary the values of $b$ and $\rho$. The rejection probabilities are computed based on 300 random samples for each value of pair $(b,\rho)$.

The simulation report in Table \ref{tbl:nonneg-cone} shows that the type I error is controlled below the target level $\alpha = 0.05$. Per statistical power, the method achieves power at least \rev{ $99\%$ for $|b| \ge 0.4$. 
Note that we have  a very difficult alternative in the sense that only a small fraction of the coordinates $({s_0}/{d})$ is negative with small magnitudes ranging in $[b/10,b]$}, so it is a very mild violation of the null, yet our algorithm still has high power.

\subsection{Real data experiment}
We measure the performance of our testing procedure on a riboflavin data set, which is publicly available by~\cite{BuhlmannBio} and can be downloaded via the `${{\sf hdi}}$' R-package. 
The data set includes $p = 4088$ predictors corresponding to the genes and $n= 71$ samples. The response variable indicates the logarithm of the riboflavin production rate and the covariates are the logarithm of the expression levels of the genes. We model the riboflavin production rate by a linear model. We first fit the Lasso solution $\hth$ using the {{\sf glmnet}} package~\cite{glmnet} and then generate $N= 100$  instances of the problem as $y^{(i)} = X\hth + w^{(i)}$, where $w^{(i)}\sim \normal(0, \sigma^2 \id_n)$. In other words, we treat $\hth$ as the true parameter $\theta_0$ and generate new data by resampling the noise.

We run two sets of experiments on this data.
\smallskip

\noindent{\bf CI for predictions.} We fix a vector $\xi\in \reals^p$ that is generated as $\xi_i \sim \normal(0,1/\sqrt{p})$, independently for $i\in [p]$. On each problem instance $(i)$, we construct confidence interval ${{\rm CI}}^{(i)}$ for $\xi^T\theta_0$, using Lemma~\ref{lem:prediction}.
We compute the coverage rate as 
\begin{align}\label{eq:cov}
{{\sf Cov}} = \frac{1}{N} \sum_{i=1}^N \ind(\xi^T\theta_0\in {{\rm CI}}^{(i)})\,.
\end{align}

\noindent{\bf CI for squared norm.} On each problem instance $(i)$, we construct confidence interval for $\|\theta_0\|_2^2$, using Lemma~\ref{lem:quadratic-CI} and compute the coverage rate given by~\eqref{eq:cov}.

The results are reported in Table~\ref{tbl:real}. As we see for various values of noise standard deviation $\sigma$, the coverage rates of the constructed intervals remain close to the nominal value.
In Figure~\ref{fig:CI-real}, we depict the constructed confidence intervals for $40$ random problem instances, in each experiment.

\begin{figure}[]
    \centering
    \subfigure[Confidence intervals for $\xi^T\theta_0$]{
        \includegraphics[width = 6.5cm]{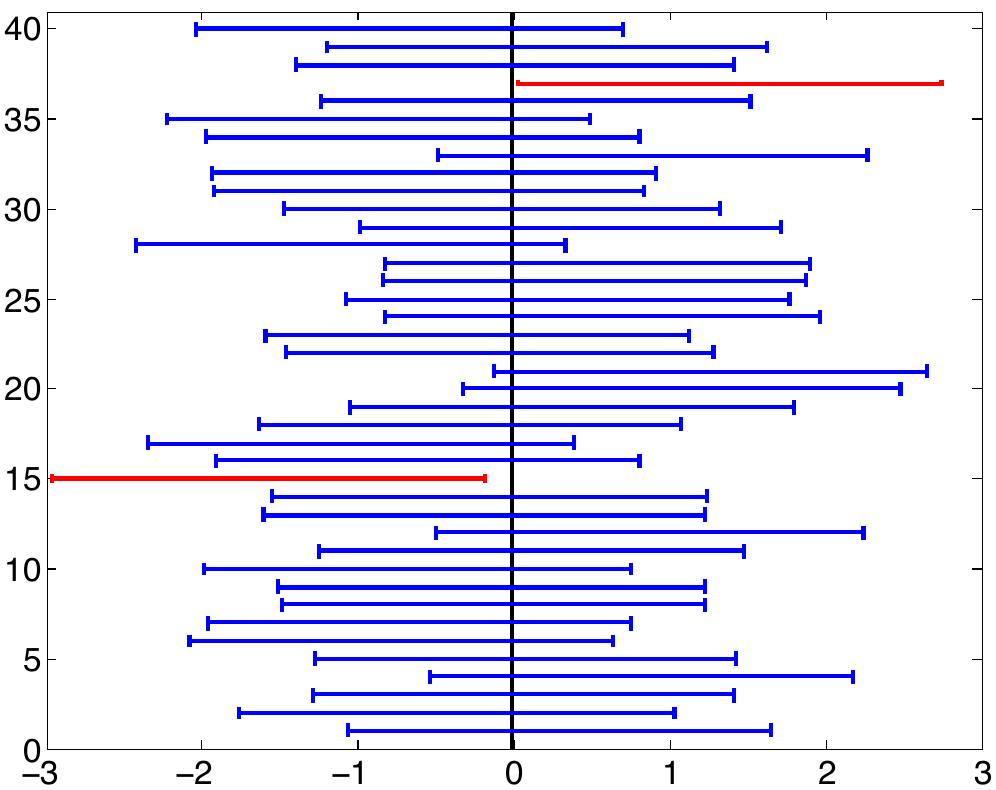}
        }
        \hspace{1cm}
    \subfigure[Confidence intervals for $\|\theta_0\|_2^2$]{
        \includegraphics[width=6.5cm]{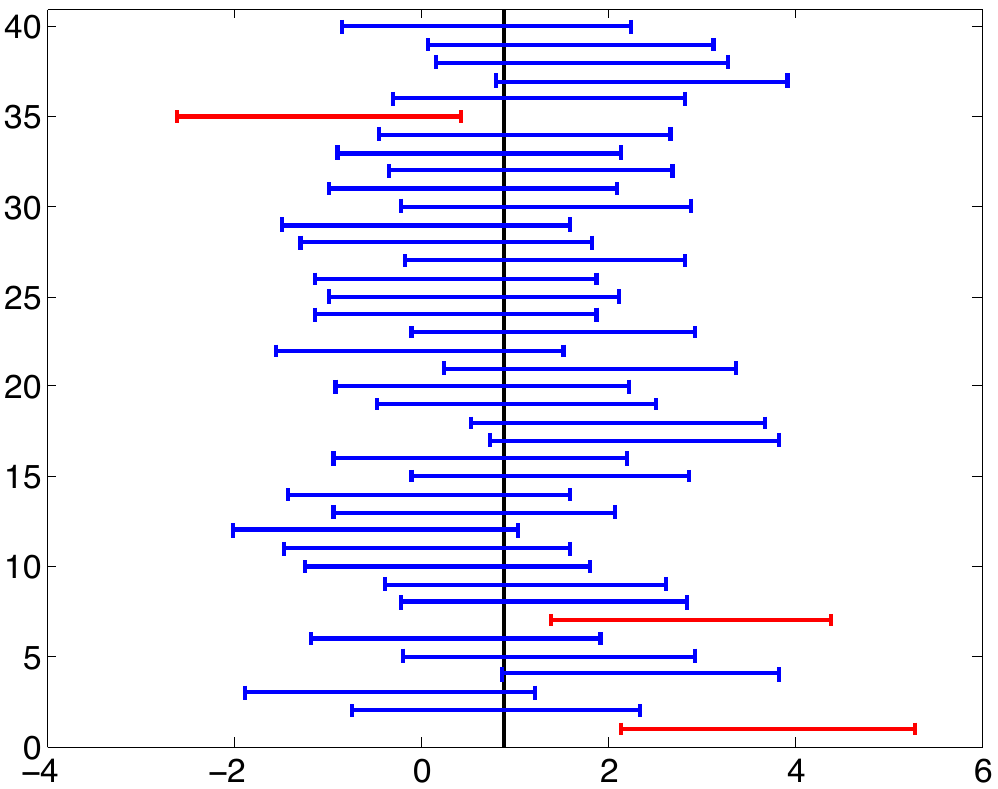}
        }
    \caption{(a) $95\%$ confidence intervals for $\xi^\sT \th_0$ (left panel) and $\|\theta_0\|_2^2$ (right panel) for riboflavin data set. The value of $\xi^\sT \theta_0$ and $\|\theta_0\|_2^2$ are indicated by the black line. A blue confidence interval covers the true value while a red one means otherwise.}\label{fig:CI-real}
\end{figure} 

\begin{table}[t]%
  \centering
  \begin{tabular}{|l|ccc|}
    \hline
    
    $\sigma$ & $1$ & $5$ & $10$ \\
    \hline
    $\xi^\sT\theta_0$ &0.96 &0.94&0.93 \\
    \hline
    $\|\theta_0\|_2^2$ &0.95 &0.93&0.94 \\

    \hline
  \end{tabular}
 \caption{Coverage rate of the confidence intervals for $\xi^\sT\theta_0$ and $\|\theta_0\|_2^2$ computed as in~\eqref{eq:cov} for the real data experiment and at various noise levels $\sigma$.}%
  \label{tbl:real}%
\end{table}


\section{Proof of Theorems}\label{proof:theorems}

\subsection{Proof of Theorem~\ref{thm:typeI}}\label{proof:typeI}
We first prove a lemma to bound the estimation error of $\hsigma$ returned by the scaled Lasso.  The following lemma uses the analysis of~\cite{SZ-scaledLasso} and its proof is given in Appendix~\ref{app:ssLasso} for reader's convenience. 
\begin{lemma}\label{ssLasso}
Under the assumptions of Theorem~\ref{thm:typeI}, let $\hsigma = \hsigma(\lambda)$ be the scaled Lasso estimator of the noise level, with $\lambda = c\sqrt{(\log p)/n}$ and define $\sigma_*  = \|w\|/\sqrt{n}$. 
Then, $\hsigma$ satisfies
\begin{align}
\prob\left(\Big|\frac{\hsigma}{\sigma^*} - 1\Big|\ge \frac{2c}{\phi_0\sigma^*} \sqrt{\frac{s_0\log p}{n}}\right)\le 2p^{-c_0}+2e^{-n/16}\,, \quad c_0 = \frac{c^2}{32K}-1\,.
\end{align}
\end{lemma} 


Armed with Lemmas~\ref{ssLasso} and \ref{propo:bias-size2} we are ready to prove Theorem~\ref{thm:typeI}.  Under $H_0$, we have $\th_0\in \Omega_0$ and hence by invoking Lemma~\ref{propo:bias-size2}, we have
\begin{align}
T_n &=\|\Dn(\dgamma-U^\sT \hthp)\|_\infty \le \|\Dn(\dgamma-U^\sT\th_0)\|_\infty \nonumber \\
&\le \frac{1}{\sqrt{n}}\|\Dn Z\|_\infty + \frac{1}{\sqrt{n}} \|\Dn \Delta\|_\infty\,. \label{eq:T_n}
\end{align}
Note that for $\tZ \equiv \hsigma \Dn Z/(\sigma\sqrt{n}) \in \reals^k$, we have $\tZ_i \sim\normal(0,1)$. The entries of $\tZ$ are correlated though.


Fix $\epsilon>0$ and apply Equation \eqref{eq:T_n} to write 
\begin{align}\label{Tn-tail}
 &\prob(T_n\ge x)\le \prob\left(\frac{\sigma}{\hsigma} \|\tZ\|_\infty + \frac{1}{\sqrt{n}} \|\Dn \Delta\|_\infty \ge x \right)\nonumber\\
&\le \prob\left(\frac{\sigma}{\hsigma} \|\tZ\|_\infty \ge x - \epsilon  \right) + \prob\left(\frac{1}{\sqrt{n}} \|\Dn \Delta\|_\infty \ge \epsilon \right)\nonumber\\
&\le \prob\left( \|\tZ\|_\infty \ge (1-\epsilon)(x - \epsilon) \right) + \prob\left(\left|\frac{\hsigma}{\sigma} - 1\right| \ge \epsilon \right) + \prob\left(\frac{1}{\sqrt{n}} \|\Dn \Delta\|_\infty \ge \epsilon \right)
\end{align}
For the second term, we proceed as follows
\begin{align}
\prob\left(\left|\frac{\hsigma}{\sigma} - 1\right| \ge \epsilon \right) \le \prob\left(\left|\frac{\hsigma}{\sigma^*} - 1\right| \ge \frac{\epsilon}{2} \right)
+ \prob\left(\left|\frac{\hsigma}{\sigma^*} - \frac{\hsigma}{\sigma}\right| \ge \frac{\epsilon}{2} \right)
\end{align}
Now, note that $\sigma^*\to \sigma$, in probability, as $n$ tends to infinity. Therefore, by applying Lemma~\eqref{ssLasso} and using the assumption $s_0 = o(n/\log p)$, we get
\begin{align}\label{sigma-hsigma}
 \underset{n\to \infty}{\lim\sup}\;\; \prob\left(\left|\frac{\hsigma}{\sigma} - 1\right| \ge \epsilon \right) = 0\,.
\end{align}
Using this in~\eqref{Tn-tail}, we have
\begin{align}
 \underset{n\to \infty}{\lim\sup}\;\;\prob(T_n\ge x) \le  \underset{n\to \infty}{\lim\sup}\;\;\prob\left( \|\tZ\|_\infty \ge (1-\epsilon)(x - \epsilon) \right)\nonumber\\
  + \underset{n\to \infty}{\lim\sup}\;\; \prob \left(\frac{1}{\sqrt{n}} \|\Dn \Delta\|_\infty \ge \epsilon  \right)
\end{align}


We next note that by definition~\eqref{Q},  and using the assumption $\lim\inf_{n\to\infty} \min_{i\in [k]} (G^\sT\hSigma G)_{ii} \ge c_0>0$, we have 
%
from which we obtain
\begin{align}
\lim\sup_{n\to\infty} \prob\left(\frac{1}{\sqrt{n}} \|\Dn \Delta\|_\infty \ge \epsilon \right) &\le \lim\sup_{n\to\infty} \prob \left(\frac{1}{\hsigma\sqrt{c_0}} \|\Delta\|_\infty \ge \epsilon \right)\nonumber\\
&\le \lim\sup_{n\to\infty}  \prob \left(\frac{2}{\sigma\sqrt{c_0}} \|\Delta\|_\infty > \epsilon\right) + \prob\left(\frac{\sigma}{\hsigma} \ge 2 \right)\,.
\end{align}

By Equation~\eqref{sigma-hsigma}, we have $\prob((\sigma/\hsigma) \ge 2)\to 0$. 
 In addition, since $s_0 = o(1/(\mu \sqrt{\log p}))$, for $n$ and $p$ large enough, we have $c\mu s_0 \sqrt{\log p}/\phi_0^2\le \epsilon\sqrt{c_0}/2$. Hence by \eqref{eq:delta-small},
\begin{align}
 \underset{n\to \infty}{\lim\sup}\;\; \prob\left(\frac{1}{\sqrt{n}} \|\Dn \Delta\|_\infty \ge \epsilon \right) 
 &\le
 \underset{n\to \infty}{\lim\sup}\;\; \prob \left(\|\Delta\|_\infty > \frac{\epsilon\sigma\sqrt{c_0}}{2} \right) \nonumber\\
&\le 
  \underset{n\to \infty}{\lim\sup}\;\; (2p^{-c_0}+2e^{-n/16}) = 0\,.\label{DnLimit}
\end{align}
By substituting~\eqref{DnLimit} in~\eqref{Tn-tail}, we get  
\begin{align}
 \underset{n\to \infty}{\lim\sup}\;\; \prob(T_n\ge x) \le  \underset{n\to \infty}{\lim\sup}\;\; \prob(\|\tZ\|_\infty \ge x-\epsilon x +\epsilon^2).\label{eq:asymp0}
\end{align}
By union bounding over the entries of $\tilde{Z}$, we get
\begin{align}
\prob(\|\tZ\|_\infty \ge x-\epsilon x +\epsilon^2) \le    2k(1-\Phi(x-\epsilon x +\epsilon^2)).
\end{align}
Observe that the above holds for any $\epsilon>0$, and that the right-hand side is bounded pointwise for all $\epsilon$. Therefore, by applying the dominated convergence theorem, we get
\begin{align*}
 \underset{n\to \infty}{\lim\sup}\;\; \prob(T_n\ge x)  \le 2k ( 1- \Phi(x)). 
\end{align*}
The result follows by choosing $x = \Phi^{-1}(1-\alpha/(2k))$.
\subsection{Proof of Theorem~\ref{thm:typeI-random}}\label{proof:typeI-random}
For $\phi_0,s_0,K\ge 0$, let 
$\event_n=\event_n(\phi_0,s_0,K)$ be the event that the compatibility
condition holds for $\hSigma=(X^{\sT}X/n)$,  
for all sets $S\subseteq [p]$, $|S|\le s_0$  with constant $\phi_0>0$, and
that $\max_{i\in [p]}\, \hSigma_{i,i} \le K$. Explicitly 
\begin{align}
\event_n(\phi_0,s_0,K) \equiv\Big\{X\in\reals^{n\times
  p}:\,\;\min_{S:\; |S|\le s_0}\phi(\hSigma,S)\ge \phi_0,\;\; 
\max_{i\in [p]}\, \hSigma_{i,i} \le K, \;\; \hSigma = (X^{\sT}X/n)\Big\}\, .
\end{align}
Then, by result of~\cite[Theorem 6]{rudelson2013reconstruction} (see also~\cite[Theorem 2.4(a)]{javanmard2014confidence}), random designs satisfy the compatibility condition with
constant $\phi_0 = \sqrt{C_{\min}}/2$, provided that $n\ge \nu s_0 \log(p/s_0)$, where $\nu = c \kappa^4(C_{\max}/C_{\min})$, for a constant $c>0$. More precisely,
\begin{align}\label{En-prob}
\prob(X \in \event_n(\sqrt{C_{\min}}/2,s_0,K) )\ge 1- 4e^{-c_1n/\kappa^4}\,,
\end{align}
where $c_1 = c_1(c)>0$ is a constant.

We next provide an explicit upper bound for the minimum generalized coherence $\mu_{\min}(X;U)$ (cf. Definition~\ref{mu*}) for random designs. 

\begin{propo}[~\cite{javanmard2014confidence}]\label{pro:mu-random}
Let $\Sigma \in \reals^{p\times p}$ be such that $\sigma_{\min}(\Sigma) \ge C_{\min}>0$ and $\sigma_{\max}(\Sigma) \le C_{\max} <\infty$ and $\max_{i\in[p]} \Sigma_{ii}\le 1$. 
Suppose that $X\Sigma^{-1/2}$ has independent sub-gaussian rows, with mean zero and sub-gaussian norm $\|\Sigma^{-1/2} x_1\|_{\psi_2} = \kappa$, for some constant $\kappa>0$.
For $U\in \reals^{p\times k}$ independent of $X$ satisfying $U^\sT U = I$, and for fixed constant $a>0$, define
\begin{align}\label{def:Gn}
\cG_n(a) \equiv \Big\{X\in \reals^{n\times p}:\, \mu_{\min}(X;U) < a\sqrt{\frac{\log p}{n}} \Big\}\,.
\end{align}
In other words, $\cG_n(a)$ is the event that problem~\eqref{OPT:M-n} is feasible for $\mu = a\sqrt{(\log p)/n}$. Then, for $n\ge a^2 C_{\min} \log p/(4e^2 C_{\max}\kappa^4)$, the following holds true
with high probability
\begin{align}\label{Gn-prob}
\prob(X\in \cG_n(a)) \ge 1- 2p^{-c_2}\,,\quad c_2  = \frac{a^2 C_{\min}}{24 e^2 \kappa^4 C_{\max}}-2.
\end{align}
\end{propo}
We refer to Appendix~\ref{proof:mu-random} for the proof of Proposition~\ref{pro:mu-random}.
\rev{

The last step is to prove that Assumption~\ref{ass:var} holds. In doing that, we use Lemma~\ref{lem:sufficient}. Note that the first condition of this lemma holds by assumption of the theorem. To prove the second condition, we use the following result.
\begin{lemma}\label{lem:concentration}
Let $\Sigma\in \reals^{p\times p}$ such that $\sigma_{\min}(\Sigma) \ge C_{\min}>0$ and $\sigma_{\max}(\Sigma) \le C_{\max} <\infty$ and $\max_{i\in[p]} \Sigma_{ii}\le 1$. 
Suppose that $X\Sigma^{-1/2}$ has independent sub-gaussian rows, with mean zero and sub-gaussian norm $\|\Sigma^{-1/2} x_1\|_{\psi_2} = \kappa$, for some constant $\kappa>0$.
Let $\hSigma\equiv (X^\sT X)/n$. For $u_i\in \reals^p$ independent of $X$, we have
\begin{align}\label{eq:quadratic}
\prob\left(u_i^\sT (\hSigma - \Sigma) u_i \ge C\sqrt{\frac{\log p}{n}} \right)\le p^{-c}\,,
\end{align}
for a constant $C>0$ depending on $\kappa$, $C_{\max}$, and $c>2$ depending on $C$.
\end{lemma}
We refer to Appendix~\ref{proof:lem-concentration} for the proof of Lemma \ref{lem:concentration}. The second condition of Lemma \ref{lem:sufficient} follows from $u_i^\sT \Sigma u_i\le C_{\max}\|u_i\|^2 = C_{\max}$,  union bounding over $i\in [k]$ and Lemma~\ref{lem:concentration} (along with Borel-Cantelli Lemma).
}

%


Putting the three probabilistic bounds~\eqref{En-prob},~\eqref{Gn-prob} and \eqref{eq:quadratic} together in Theorem~\ref{thm:typeI}, we obtain that for random designs with $s_0 = o(\sqrt{n}/(\log p))$, we have $\underset{n\to \infty}{\lim\sup}\;\; \alpha_n(R_X) \le \alpha$.

\subsection{Proof of Theorem~\ref{thm:power}}\label{proof:power}
We start by stating a lemma that will be used later in the proof. 
\rev{
\begin{lemma}\label{lem:random-deter}
Under the assumptions of Theorem~\ref{thm:typeI-random},  for any $ i\in [k]$ we have 
\begin{align*}
\prob\Big(g_i^\sT \hSigma g_i \ge u_i^\sT\Sigma^{-1} u_i+C \sqrt{\frac{\log p}{n}}\Big)\le
2\, p^{-c}\, ,
\end{align*}
where $c$ is a constant depending on $a, C$ and by a suitable choice of them, we have $c\ge 2$. 

\end{lemma}
We refer to Appendix~\ref{proof:random-deter} for the proof of Lemma~\ref{lem:random-deter}. 
\begin{coro}
Assuming the setting of  Theorem~\ref{thm:typeI-random}, by an application of Borel-Cantelli lemma and using Lemma~\ref{lem:random-deter}, of any $i\in [k]$ we have almost surely
\begin{align}
\lim\sup_{n\to \infty} [g_i^\sT \hSigma g_i - u_i^\sT \Sigma^{-1}u_i]\le 0\,.
\end{align}
\end{coro}
}
Recalling the definition of $m_0$, given by~\eqref{def:m0}, we have the following corollary.  

\begin{coro}\label{coro:m0}
Recalling the definition of $m_0$ given by~\eqref{def:m0}, for any $i\in[k]$, we have almost surely
\begin{align}
\lim\sup_{n\to \infty} [g_i^\sT \hSigma g_i - m_0^2]\le 0\,.
\end{align}
\end{coro}

Let $z_* \equiv \Phi^{-1}(1-\alpha/(2k))$ and write
\begin{align}
&\underset{n\to \infty}{\lim\inf}\;\; \frac{1-\beta_n(R_X)}{1-\beta_n^*(\eta)}\nonumber\\
&= \underset{n\to \infty}{\lim\inf}\;\; \frac{1}{1-\beta_n^*(\eta)}\; \inf_{\th_0} \Big\{\prob_{\th_0}(R_X =1):\; \|\th_0\|_0\le s_0,\; \de(\th_0,\Omega_0) \ge \eta \Big\}\nonumber\\
&= \underset{n\to \infty}{\lim\inf}\;\; \frac{1}{1-\beta_n^*(\eta)}\; \inf_{\th_0} \Big\{\prob\Big(\|\Dn(\dgamma-U^\sT \hthp)\|_\infty \ge z_* \Big):\; \|\th_0\|_0\le s_0,\; \de(\th_0,\Omega_0) \ge \eta \Big\}\label{eq:P0}
\end{align}
We define the shorthands $v\equiv \Dn U^\sT(\hthp-\th_0)$ and $\tilde{v} \equiv\Dn(\dgamma-U^\sT\th_0)$. Note that $v, \tilde{v} \in \reals^k$. We further let $i^\star \equiv \arg\max_{i\in [k]} |v_i|$. Then, we can write 
\begin{align}\label{eq:v-1}
\|\Dn(\dgamma-U^\sT \hthp)\|_\infty = |v-\tilde{v}|_\infty \ge |v_{i^\star} - \tilde{v}_{i^\star}| 
\end{align}
By a very similar argument we used to derive Equation~\eqref{eq:asymp0}, we can show that for any fixed $i\in [k]$ and all $x\in \reals$, we have
\begin{align}\label{normality}
\lim\sup_{n\to \infty} \sup_{\|\theta_0\|_0\le s_0} |\prob(\tilde{v}_i \le x) \le \Phi(x)| = 0\,.
\end{align}
In words, each coordinate of $\tilde{v}$ asymptotically admits a standard normal distribution.

The other remark we want to make is about the quantity $\|{v}\|_\infty$, which will be a key factor in determining the power of the test. Because $\hthp\in\Omega_0$, we have 
\begin{align}\label{vstarB}
|{v}_{i^\star}| = \|{v}\|_\infty \ge \min_{i\in [k]}(D_{ii})\; \|U^\sT(\hthp-\th_0)\|_\infty \ge \min_{i\in [k]}(D_{ii})\; \de(\th_0,\Omega_0) \ge \eta \min_{i\in [k]}(D_{ii})\,.
\end{align}

Continuing with \eqref{eq:P0}, we write
\begin{align}
&\underset{n\to \infty}{\lim\inf}\;\; \frac{1-\beta_n(R_X)}{1-\beta_n^*(\eta)}\nonumber\\
&=\underset{n\to \infty}{\lim\inf}\;\; \frac{1}{1-\beta_n^*(\eta)} \inf_{\th_0} \Big\{\prob\Big(\|\Dn(\dgamma-U^\sT \hthp)\|_\infty \ge z_* \Big):\; \|\th_0\|_0\le s_0,\; \de(\th_0,\Omega_0) \ge \eta \Big\}\nonumber\\
&\stackrel{(a)}{\ge} \underset{n\to \infty}{\lim\inf}\;\; \frac{1}{1-\beta_n^*(\eta)} \inf_{\th_0} \Big\{\prob\left(|v_{i^\star} - \tilde{v}_{i^\star}| \ge z_* \right):\; |v_{i^\star}| \ge \eta  \min_{i\in [k]}(D_{ii}) \Big\}\nonumber\\
&= \underset{n\to \infty}{\lim\inf}\;\; \frac{1}{1-\beta_n^*(\eta)} \left(1-\sup_{\th_0} \Big\{\prob\left(|v_{i^\star} - \tilde{v}_{i^\star}| \le z_* \right):\; |v_{i^\star}| \ge \eta  \min_{i\in [k]}(D_{ii}) \Big\} \right) \nonumber\\
&\ge \underset{n\to \infty}{\lim\inf}\;\; \frac{1}{1-\beta_n^*(\eta)} \left(1-\sup_{\th_0} \Big\{\prob\left(\exists j\in [k]: |v_{i^\star} - \tilde{v}_{j}| \le z_* \right):\; |v_{i^\star}| \ge \eta  \min_{i\in [k]}(D_{ii}) \Big\} \right) \nonumber\\
&\ge \underset{n\to \infty}{\lim\inf}\;\; \frac{1}{1-\beta_n^*(\eta)} \left(1-k \sup_{\th_0} \Big\{\prob\left(|v_{i^\star} - \tilde{v}_{1}| \le z_* \right):\; |v_{i^\star}| \ge \eta  \min_{i\in [k]}(D_{ii}) \Big\} \right) \nonumber\\
&\stackrel{(b)}{\ge} \underset{n\to \infty}{\lim\inf}\;\; \frac{1}{1-\beta_n^*(\eta)} \left(1-k  \prob\Big(\Big|\frac{\sqrt{n}\eta}{\hsigma m_0} - Z\Big| \le z_* \Big) \right) \nonumber\\
& = \underset{n\to \infty}{\lim\inf}\;\; \frac{1}{1-\beta_n^*(\eta)} \Big(1-k\,\Big\{\Phi\Big(\frac{\sqrt{n}\eta}{\hsigma m_0}+z_*\Big)- \Phi\Big(\frac{\sqrt{n}\eta}{\hsigma m_0}-z_*\Big)\Big\}\Big)\nonumber\\
& \stackrel{(c)}{=} \underset{n\to \infty}{\lim\inf}\;\; \frac{1}{1-\beta_n^*(\eta)} \myF\Big(\alpha,\frac{\sqrt{n}\eta}{\hsigma m_0}, k\Big) = 1\,,
\end{align}
where $(a)$ follows from Equations~\eqref{eq:v-1} and~\eqref{vstarB}; $(b)$ holds because of Corollary~\ref{coro:m0} and Equation~\eqref{normality}. Here $Z$ is a standard normal variable; $(c)$ follows by substituting for $z_*$. 

\rev{
\subsection{Proof of Theorem~\ref{thm:typeI-random-aps}} \label{proof:thm-typeI-random-aps}
The proof goes along the same lines of the proof of Theorem~\ref{thm:typeI} and \ref{thm:typeI-random}. 

Defining $r = X\theta_*-X \theta_0$ and by plugging in for $y = X\theta_* +w = X\theta_0+ r + w$ in the definition~\eqref{dgamma}, we get
\begin{align}
\dgamma &= U^\sT \htheta + \frac{1}{n} G^\sT X^\sT X(\theta_0 - \htheta) + \frac{1}{n} G^\sT X^\sT r + \frac{1}{n} G^\sT X^\sT w \nonumber\\
& = U^\sT \theta_0 + (G^\sT \hSigma - U^\sT) (\theta_0 - \htheta) +   \frac{1}{n} G^\sT X^\sT r + \frac{1}{n} G^\sT X^\sT w \nonumber\\
& = U^\sT \theta_0 + \frac{1}{\sqrt{n}} \Delta + \frac{1}{\sqrt{n}} Z\,,
\end{align}
with 
\[
\Delta \equiv \Delta_1 + \Delta_2\,, \quad \Delta_1\equiv \sqrt{n} (G^\sT \hSigma - U^\sT) (\theta_0 - \htheta)\,,\quad  \Delta_2 \equiv   \frac{1}{\sqrt{n}} G^\sT X^\sT r\,, \quad
Z \equiv   \frac{1}{\sqrt{n}} G^\sT X^\sT w\,.
\]
Sine $w\sim \normal(0,\sigma^2\id_{n\times n})$, we have $Z|X\sim \normal(0,\sigma^2G^\sT\hSigma G)$. We next bound $\|\Delta\|_\infty$. 

It is straightforward to see that the assumptions of Theorem~\ref{thm:typeI-random-aps} implies the assumption of Lemma~\ref{lem:moment} and hence by the result of the lemma, the moment conditions (Assumption~\ref{ass:moment}) hold. 
To deal with $\Delta_1$, we use the following result from \cite{belloni2012sparse} that bounds the $\ell_1$ error of the iterated Lasso estimator under the Assumptions~\ref{ass:app-sparsity} and \ref{ass:moment}.
\begin{propo}\label{propo:aps} (\cite[Theorem 1]{belloni2012sparse})
Suppose that in the regression model~\eqref{eq:app-sparsity-model}, Assumption~\ref{ass:app-sparsity} (approximate sparsity) and Assumption~\ref{ass:moment} (Moment Conditions) hold. Let $\hth$ be the iterated lasso estimator~\eqref{eq:pen} with weights $\gamma_j$ specified by Algorithm~\ref{lambda:iterated}. Then, $\hth$ satisfies
\begin{align}
\|\hth - \theta_0\|_1\le C C_{\min}^{-1} s_0 \sqrt{\frac{\log p}{n}} \,,
\end{align}
with high probability, for some finite constant $C>0$.
\end{propo}

Now let $\cE_n$ be the probability event that $\|\hth - \theta_0\|_1 \le CC_{\min}^{-1} s_0 \sqrt{(\log p)/n}$. Recall the event $\cG_n(a)$ from~\eqref{def:Gn} and define $\cF_n \equiv \cG_n(a) \cap \cE_n$. Then, by using Propositions~\ref{pro:mu-random} and \ref{propo:aps}, we have the $\cF_n$ happens with high probability. Further, on the event $\cF_n$ we have
\begin{align}\label{Delta-1-B}
\|\Delta_1\| \le \sqrt{n} \times a\sqrt{\frac{\log p}{n}} \times C C_{\min}^{-1} s_0 \sqrt{\frac{\log p}{n}} =  C C_{\min}^{-1} a s_0 \frac{\log p}{\sqrt{n}}\,.
\end{align}
We next bound $\Delta_2$. Write
\begin{align*}
\|\Delta_2\|_\infty \le \left(\max_{i\in [k]} \Big\|\frac{1}{\sqrt{n}} X g_j \Big\| \right) \|r\|\,.
\end{align*}
Using lemma~\ref{lem:random-deter}, we have
\[
\Big\|\frac{1}{\sqrt{n}} Xg_i \Big\|^2 = g_i^\sT\hSigma g_i \le u_i^\sT \Sigma^{-1} u_i + C\sqrt{\frac{\log p}{n}} \le \frac{1}{C_{\min}} + C\sqrt{\frac{\log p}{n}} < C'\,,
\]
with $C' = 1/C_{\min}+ C$, and with probability at least $1 - 2p^{-c}$, for $c\ge 2$. By union bounding over $i\in [k]$, we get 
\[
\max_{i\in [k]}\, \Big\|\frac{1}{\sqrt{n}}Xg_i\Big\| \le C'\,,
\]
with probability at least $1 - 2kp^{-c} \ge 1 - 2p^{-c+1}$. Using Assumption~\ref{ass:app-sparsity}, $\|r\| = o_P(1)$, which gives us
\begin{align}\label{Delta-2-B}
\|\Delta_2\|_\infty = o_P(1)\,.
\end{align}
Combining~\eqref{Delta-1-B} and \eqref{Delta-2-B}, we have
\[
\|\Delta\|_\infty = O_P\Big(s_0 \frac{\log p}{\sqrt{n}}\Big) + o_P(1)\,.
\] 
Hence $\|\Delta\|_\infty = o_p(1)$ and $Z|X$ is asymptotically normally distributed. Having this result, we can then follows the lines of the proof of Theorem~\ref{thm:typeI-random} to show that our procedure controls the type I error, that is $\lim\sup_{n\to\infty} \alpha_n(R_X)\le \alpha$.
}
\section*{Acknowledgements}
A. Javanmard was partially supported by an Outlier Research in Business (iORB) grant from the USC Marshall School of Business, a Google Faculty Research award and the NSF CAREER award DMS-1844481.

\newpage
\appendix
\section{Proof of Technical Lemmas}
\rev{
\subsection{Proof of Lemma~\ref{lem:sufficient}}\label{app:lem-sufficient}
We start by providing a non-asymptotic lower bound on $(G^\sT \hSigma G)_{i,i}$. 
\begin{lemma}\label{lem:LB-opt}
Let $G$ be the matrix with rows $g_i^\sT$ obtained by solving optimization~\eqref{OPT:M-n}. Then, for all $i\in [p]$,
\[
(G^\sT \hSigma G)_{i,i} \ge  \frac{(1-\mu\|u_i\|_1)^2}{u_i^\sT\hSigma u_i}\,.
\]
\end{lemma}
Using this lemma we write
\begin{align*}
\lim\inf_{n\to\infty} \min_{i\in[k]}\, (G^\sT \hSigma G)_{i,i} &\ge \lim\inf_{n\to\infty} \min_{i\in[k]}  \frac{(1-\mu\|u_i\|_1)^2}{u_i^\sT\hSigma u_i}\\
&\ge \left(\lim\inf_{n\to\infty} \min_{i\in[k]}\, (1-\mu\|u_i\|_1)^2 \right)  \left(\lim\sup_{n\to\infty} \max_{i\in[k]} u_i^\sT\hSigma u_i \right)^{-1} \\
&\ge \left(\lim\inf_{n\to\infty}\, (1-\mu \max_{i\in[k]} \|u_i\|_1)^2 \right) \times C^{-1}\\
&\ge (1-c)^2 C^{-1}\,,
\end{align*}
which completes the proof.

\subsubsection{Proof of Lemma~\ref{lem:LB-opt}}
The proof proceeds as the proof of~\cite[Lemma 3.1]{javanmard2014confidence}. Let $C_i(\mu)$ be the solution of optimization~\eqref{OPT:M-n}.  We write
\begin{align*}
\<u_i,u_i - \hSigma g\> \le \|u_i\|_1 \|u_i - \hSigma g\|_\infty \le \mu \|u_i\|_1\,. 
\end{align*}
Hence, for feasible $\tilde{g}$ and any $c\ge 0$, and by using that $\|u_i\| = 1$ for $i\in [k]$, 
\[
\tilde{g}^\sT \hSigma \tilde{g} \ge \tilde{g}^\sT \hSigma \tilde{g} + c(1 - \mu\|u_i\|_1)  - cu_i^\sT \hSigma \tilde{g} \ge \min_g\Big\{{g}^\sT \hSigma {g} + c(1 - \mu\|u_i\|_1)  - cu_i^\sT \hSigma {g} \Big\}\,.
\]
Then by minimizing over all feasible $\tilde{g}$, 
\[
C_i(\mu) \ge  \min_g\Big\{{g}^\sT \hSigma {g} + c(1- \mu\|u_i\|_1)  - cu_i^\sT \hSigma {g} \Big\}\,.
\]
The minimum of the right hand side is achieved for $g = cu_i/2$ which implies that 
\[
C_i(\mu) \ge c(1 - \mu\|u_i\|_1) - \frac{c^2}{4} (u_i^\sT \hSigma u_i)\,.
\]
The claim follows by optimizing over $c\ge 0$.

\subsection{Proof of Lemma~\ref{lem:concentration}}\label{proof:lem-concentration}
Fix $i\in [k]$ and write $u_i^\sT \hSigma u_i = \frac{1}{n}\sum_{\ell=1}^n (u_i^\sT x_\ell)^2$. Let $V_\ell = u_i^\sT x_\ell$ then $\E(V_\ell^2) = u_i^\sT \Sigma u_i$. Further, using the sub-gaussian assumption on the covariates $x_i$, we have
\[
\|V_\ell\|_{\psi_2} \le \|\Sigma^{1/2} u_i\|_2 \|\Sigma^{-1/2} x_\ell\|_{\psi_2} \le {\kappa}\sqrt{C_{\max}}\,.
\]
Let $S_\ell = V_\ell^2 - u_i^\sT \Sigma u_i$. Then $S_\ell$ is zero mean and its sub-exponential norm can be bounded as $\|S_\ell\|_{\psi_1} \le 2 \|V_\ell^2\|_{\psi_1} \le 2 \|V_\ell^2\|_{\psi_1}\le 4\|V_\ell\|_{\psi_2}^2 \le 4 \kappa^2 C_{\max} \equiv C'$. Therefore, by an application of Bernstein inequality for centered sub-exponential random variables~\cite{Vershynin-CS} (similar to the
proof of Lemma \ref{lem:main_lem}), we have that for $\eps\le eC'$\,,
\[
\prob\left( u_i^\sT \hSigma u_i\ge u_i^\sT \Sigma u_i +\eps \right) \le \exp\Big[-\frac{n}{6}\min\Big((\frac{\eps}{eC'})^2, \frac{\eps}{eC'} \Big) \Big] \,.
\] 
For $\eps = C\sqrt{(\log p)/n}$ and assuming $n\ge [C/(eC')]^2 \log p$, we obtain 
\[
\prob\left( u_i^\sT \hSigma u_i\ge u_i^\sT \Sigma u_i + C\sqrt{\frac{\log p}{n}} \right) \le p^{-C^2/(6e^2C'^2)} \,.
\]
The result follows.
}

\subsection{Proof of Lemma~\ref{lem:k1}}\label{proof:lem:k1}
Consider the following two optimization problems:
\begin{align*}
\underset{{k\in [p], U\in \reals^{p\times k}}}{\text{maximize}} \quad &\myF\left(\alpha,\frac{1}{\hsigma} \sqrt{n C_{\min}}\, \de(\hth^{(1)},\Omega_0;U), k\right) \quad \text{ subject to }\quad  \; U^\sT U = \id_k\,. \, \;\;\quad\quad\quad \quad\quad\quad ({{\rm P}}_1)\\
\underset{{u\in \reals^{p\times 1}}}{\text{maximize}} \quad &\myF\left(\alpha,\frac{1}{\hsigma} \sqrt{n C_{\min}}\, \de(\hth^{(1)},\Omega_0;u), 1\right) \quad \text{ subject to } \quad\|u\|_2 = 1\,.  \quad \quad  \quad\quad
\quad\quad\quad\;\;\,({{\rm P}}_2)
\end{align*}
Let ${{\sf OPT}}_1$ and ${{\sf OPT}}_2$ respectively denote the optimal value of problems $({{\rm P}}_1)$ and $({{\rm P}}_2)$. Clearly ${{\sf OPT}}_1 \ge {{\sf OPT}}_2$.
We next show the reverse side.

First note that
\begin{align}\label{eq:dum1}
\inf_{\theta\in \Omega_0 }\|U^\sT(\theta - \hth^{(1)})\|_\infty = \inf_{\theta\in \Omega_0} \max_{v: \|v\|_1\le 1} v^\sT U^\sT (\theta - \hth^{(1)})\,.
\end{align}
Since the right-hand side is linear in $v$ and $\theta$, and $\Omega_0$ is convex, by Von Neumann's minimax theorem, we have
\begin{align}\label{eq:dum2}
\inf_{\theta\in \Omega_0} \max_{v: \|v\|_1\le 1} v^\sT U^\sT (\theta - \hth^{(1)}) =  \max_{v: \|v\|_1\le 1} \inf_{\theta\in \Omega_0} v^\sT U^\sT (\theta - \hth^{(1)})\,.
\end{align}
Let $\tilde{v} = U v$. Since $U$ has orthonormal columns we have $\|\tilde{v}\|_2 = \|v\|_2\le \|v\|_1 \le 1$. Using this observation along with Equations~\eqref{eq:dum1} and~\eqref{eq:dum2}, we get
\begin{align}\label{eq:dum3}
\inf_{\theta\in \Omega_0 }\|U^\sT(\theta-\hth^{(1)})\|_\infty \le  \max_{u: \|u\|_2\le 1} \inf_{\theta\in \Omega_0} u^\sT (\theta-\hth^{(1)})\,.
\end{align}
Therefore, for any $U\in \cJ$, there exists unit norm vector $u\in \reals^p$, such that
\begin{align}\label{eq:dum4}
\de(\hth^{(1)},\Omega_0;U)\le  \de(\hth^{(1)},\Omega_0;u)\,. 
\end{align}

Before we proceed with the rest of the proof we state a lemma about the function $G$.
\begin{lemma}\label{lem:G-dec}
The function $k\mapsto \myF(\alpha,x,k)$ is strictly decreasing in $k$.
\end{lemma}
Now choose any $U\in \cJ$ and choose unit norm $u$ that satisfies~\eqref{eq:dum4}. Then,
\begin{align*}
{{\sf OPT}}_1 = \myF\left(\alpha,\frac{1}{\hsigma} \sqrt{n C_{\min}}\, \de(\hth^{(1)},\Omega_0;U), k\right)
& \le \myF\left(\alpha,\frac{1}{\hsigma} \sqrt{n C_{\min}}\, \de(\hth^{(1)},\Omega_0;u), k\right)\\
&\le \myF\left(\alpha,\frac{1}{\hsigma} \sqrt{n C_{\min}}\, \de(\hth^{(1)},\Omega_0;u), 1\right)\,, 
\end{align*} 
where the first inequality follows from monotonicity of $\myF(\alpha,x,k)$ in $x$ and the second inequality follow from Lemma~\ref{lem:G-dec}.
This implies that ${{\sf OPT}}_1\le {{\sf OPT}}_2$. 

Therefore ${{\sf OPT}}_1 = {{\sf OPT}}_2$ which completes the proof. Indeed, we have proved a stronger claim that $\cJ$ only includes one-dimensional subspaces ($k=1$). This follows readily from the above proof and the fact that $\myF(\alpha,x,k)$ is \emph{strictly} decreasing in $k$ as per Lemma~\ref{lem:G-dec}.

\subsubsection{Proof of Lemma~\ref{lem:G-dec}}
Recall the definition of $\myF$ given by
\[
\myF(\alpha,x,y) = 1 - y\Big\{\Phi\Big(x+\Phi^{-1}\Big(1-\frac{\alpha}{2y}\Big) \Big) - \Phi\Big(x - \Phi^{-1}\Big(1-\frac{\alpha}{2y}\Big) \Big) \Big\}\,.
\]
Let $z = \Phi^{-1}(1-\alpha/(2y))$. We then have
\begin{align*}
\frac{\partial}{\partial y} \myF(\alpha,x,y) =  &- \Big\{\Phi(x+z) - \Phi(x - z) \Big\} - y \Big\{\frac{\varphi(x+z)}{\varphi(z)} + \frac{\varphi(x-z)}{\varphi(z)} \Big\}\frac{\alpha}{2y^2}\,,
\end{align*}
where $\varphi(t)\equiv e^{-t^2/2} \de t/\sqrt{2\pi}$ is the standard normal density function. Since $z > 0$ and $\Phi$ is monotone increasing, it is easy to see that $(\partial/\partial y) \myF(\alpha,x,y) < 0$ for $y> 0$.
%
\rev{
\subsection{Proof of Proposition~\ref{lindberg}}\label{app:lindberg}
%
Write 
\begin{align}
Z_i = \frac{1}{\sqrt{n}}\sum_{\ell=1}^n \zeta_\ell, \quad \text{ with }  \zeta_\ell \equiv \frac{g_i^\sT x_\ell w_\ell}{ \sigma (g_i^\sT \hSigma g_i)^{1/2}}\,.
\end{align}

Note that conditional on $X$, the random variables $\zeta_\ell$ are zero mean and independent. In addition, $\sum_{\ell=1}^n \E(\zeta_\ell^2|X) = n$.  Let $c_n = (g_i^\sT \hSigma g_i)^{1/2}$. Similar to the proof of Theorem \ref{thm:typeI-random}, by using Lemma~\ref{lem:concentration} and \ref{lem:sufficient}, Assumption~\ref{ass:var} holds and hence
\[
\lim\inf_{n\to\infty} c_n \ge c_0 > 0\,,
\]
for some positive constant $c_0$. We are now ready to prove that the Lindeberg condition holds. If optimization~\eqref{OPT:M-n-n} is feasible for $i\in [k]$, then $|\zeta_\ell| \le (\sigma c_n)^{-1} \|Xg_i\|_\infty \|w\|_\infty \le 
(\sigma c_n)^{-1} n^{\beta} \|w\|_\infty$. Therefore,
\begin{align*}
\lim_{n\to\infty} \frac{1}{n} \sum_{\ell=1}^n \E\left(\zeta_\ell^2\ind(|\zeta_\ell|\ge \eps\sqrt{n}) |X \right) &\le \lim_{n\to\infty} \frac{1}{n} \sum_{\ell=1}^n \E\left\{\zeta_\ell^2\ind\Big(\|w\|_\infty \ge \eps\sigma c_n n^{1/2-\beta}\Big) |X \right\}\\
& \le \lim_{n\to\infty} \frac{1}{n} \sum_{\ell=1}^n \frac{g_i^\sT x_\ell x_\ell^\sT g_i}{\sigma^2(g_i^\sT \hSigma g_i)}  \E\left\{w_\ell^2 \ind\Big(\|w\|_\infty > \eps {\sigma c_0} n^{1/2-\beta}\Big) \right\}\\
& \le \lim_{n\to\infty} \frac{1}{n} \sum_{\ell=1}^n \frac{g_i^\sT x_\ell x_\ell^\sT g_i}{\sigma^2(g_i^\sT \hSigma g_i)}  \E\left\{\|w\|_\infty^2 \ind\Big(\|w\|_\infty > \eps \eps {\sigma c_0} n^{1/2-\beta}\Big) \right\}\\
&\le \lim_{n\to \infty} \frac{n}{\sigma^2} \E\left\{w_1^2 \ind\Big(|w_1| > \eps {\sigma c_0} n^{1/2 - \beta}\Big)\right\}\\
&\le \frac{1}{\sigma^2} \Big(\frac{1}{\sigma \eps c_0}\Big)^{2+a}\lim_{n\to\infty} n^{1 - (2+a)(1/2-\beta)} \E(|w_1|^{4+a}) = 0\,, 
\end{align*}
where the last limit follows since $a> 4\beta/(1-2\beta)$ and $\E(|w_1|^{4+a})$ is finite.

What we are left to prove is that optimization~\eqref{OPT:M-n-n} is feasible for all $i\in [k]$, with high probability. This follows by showing that $\Sigma^{-1} u_i$ is a feasible solution to \eqref{OPT:M-n-n} using the tail bound inequality for sub-gaussian variables. This is very similar to the argument presented in the proof of \cite[Lemma 6.3]{javanmard2014confidence} and is omitted here. 
}
\subsection{Proof of Lemma \ref{lem:prediction}}\label{proof:lem-prediction}
By computing $u$ from~\eqref{eq:u-opt} in case of $\Omega_0 = \{\th:\, \<\xi,\th\> = c\}$, we have $u = \xi/\|\xi\|$.  Let $q= (\hsigma^2/n) (g^\sT \hSigma^{(2)} g + 10^{-4})$ and $d = q^{-1/2}$. Then, the test statistics~\eqref{test-statistic}
becomes $$T_n = \Big|d \Big(\dgamma - \frac{\xi^\sT \pth}{\|\xi\|}\Big)\Big| = \Big|d \Big(\dgamma - \frac{c}{\|\xi\|}\Big)\Big|\,,$$ because $\pth\in \Omega_0$.

By duality of hypothesis testing and confidence intervals, the $(1-\alpha)$ confidence interval of $\<\xi,\th_0\>$, denoted by $C(\alpha)$, consists of all values $c$ such that we fail to reject $H_0$ at level $\alpha$. Namely, $C(\alpha) = [c_{\min}, c_{\max}]$ such that $c\in C(\alpha)$ if and only if $T_n<z_{\alpha/2}$. Plugging for $d$ this yields
\begin{align*}
c_{\min}  &= \Big(\dgamma - \frac{\hsigma}{\sqrt{n}} \sqrt{g^\sT\hSigma g}\, z_{\alpha/2}\Big) \|\xi\|\,,\\
c_{\max} &= \Big(\dgamma + \frac{\hsigma}{\sqrt{n}} \sqrt{g^\sT\hSigma g}\, z_{\alpha/2}\Big) \|\xi\|\,.
\end{align*}
The proof is complete.
\rev{
\subsection{Proof of Lemma~\ref{lem:quadratic-CI}}\label{proof:lem-quadratic-CI}
For $\phi_0,s_0, K\in \reals_{\ge 0}$, define the set $\event_n(\phi_0,s_0,K)$ as follows:
\begin{align*}
\event_n(\phi_0,s_0,K) \equiv \{X\in\reals^{n\times p}:\quad \min_{S:|S|\le s_0} \phi(\hSigma,S)\ge \phi_0, \,\, \max_{i\in [p]} \hSigma_{i,i}\le K,\, \, \hSigma \equiv (X^\sT X)/n\}\,.
\end{align*} 
By using~\cite[Theorem 2.4 (a)]{javanmard2014confidence}, there exists constant $c_*\le 2000$, such that for $n\ge Cs_0\log(p/s_0)$, $C  = 4c_*(C_{\max} \kappa^4/C_{\min})$ and $\phi_0 = C_{\min}^{1/2}$, $K \ge 1+20\kappa^2\sqrt{(\log p)/n}$, we have
\begin{align}\label{E-event}
\prob(X\in \event_n)\ge 1-4e^{-c_1n}, \quad c_1 \equiv \frac{1}{c_*\kappa^4}\,.
\end{align}

Define the set $\Omega_1 \equiv \{\theta\in \reals^p:\, \|\hth^{(1)}-\theta\|_2\le \frac{\sqrt{20 s_0}}{\phi_0^2}\lambda\}$. Using the result of~\cite[Lemma 6.10]{buhlmann2011statistics}, we have that for $\lambda \ge 8\sigma \sqrt{K(1+c_0) (\log p)/n}$, 
\begin{align}\label{Bickel-L2}
\prob(\theta_0\in \Omega_1) \ge 1 - 2p^{-c_0}\,. 
\end{align}
(Note that $\hth^{(1)}$ is computed based on $n/2$ samples.) 

Let $\cG$ be the probability event that both of the above high probability events hold, that is $\cG\equiv  \{X\in \event_n(\phi_0,s_0,K)\} \cap \{\theta_0\in \Omega_1\}$. Therefore, $\prob(\cG)\ge  1- 4e^{-c_1n} -2p^{-c_0}$. 

Assuming $\cG$, we rewrite the $\ell_\infty$ projection in~\eqref{my-Tn} for this case with $u = \hth^{(1)}/\|\hth^{(1)}\|$ and $k = 1$, as follows:
\begin{align}
\hthp = & \quad \underset{\theta \in \reals^p}{\text{argmin}} \quad \Big |d \Big(\dgamma - \frac{ \theta^\sT \hth^{(1)}}{\|\hth^{(1)}\|}\Big)\Big|\,
\quad \text{subject to} \;\;\; \th\in \Omega_0(c)\cap \Omega_1\,, 
\end{align}
and the test statistics is given by $T_n = |d(\dgamma- (\hth^{(1)})^\sT\hthp/\|\hth^{(1)}\|)|$. By duality of hypothesis testing and confidence intervals, we need to find the range of values of $c$, such that $T_n\le z_{\alpha/2}$ (i.e, the test rule fails to reject the null hypothesis). Note that $T_n\le z_{\alpha/2}$ if and only if
\begin{align}
\big|\dgamma \|\hth^{(1)}\| - (\hth^{(1)})^\sT \pth  \big| <\frac{1}{d} z_{\alpha/2} \|\hth^{(1)}\|\,.
\end{align}
Writing $(\hth^{(1)})^\sT \pth = \tfrac{1}{2} (\|\pth\|^2 + \|\hth^{(1)}\|^2  - \|\pth - \hth^{(1)}\|^2)$ and using that fact that $\pth\in\Omega_0(c)\cap \Omega_1$, the above inequality yields
\begin{align*}
\Big|\dgamma \|\hth^{(1)}\|  -\tfrac{1}{2} \|\hth^{(1)}\|^2 - \tfrac{1}{2} c \Big| &<\frac{1}{d} z_{\alpha/2} \|\hth^{(1)}\| + \frac{1}{2} \|\hth^{(1)} - \pth\|^2 \\
&\le \frac{1}{d} z_{\alpha/2} \|\hth^{(1)}\| + \frac{10s_0}{\phi_0^4} \lambda^2 \\
&\le \frac{1}{d} z_{\alpha/2} \|\hth^{(1)}\| + C \frac{s_0\log p}{n}\,, 
\end{align*}
with $C\equiv \frac 640\sigma^2 K(1+c_0)\phi_0^{-4}$. Rearranging the terms and substituting for $d$, we get $c\in C(\alpha) = [c_{\min}, c_{\max}]$, where $c_{\min}$ and $c_{\max}$ are given by
\begin{eqnarray*}
c_{\min}&\equiv& 2 \|\hth^{(1)}\| \dgamma - \|\hth^{(1)}\| - L - C \frac{s_0\log p}{n} \,,\\
c_{\max} &\equiv& 2 \|\hth^{(1)}\| \dgamma - \|\hth^{(1)}\| + L + C \frac{s_0\log p}{n}\,,
\end{eqnarray*}
with $L$ given by~\eqref{L}. As shown in the proof of Theorem~\ref{thm:typeI}, Assumption~\ref{ass:var} holds which implies that $L\gtrsim 1/\sqrt{n}$. In addition, by our assumption $s_0 = o(\sqrt{n}/\log p)$, which results in $\frac{s_0\log p}{n} = o(L)$.

Regarding the coverage probability, we use the duality of hypothesis testing and confidence intervals to obtain 
 \begin{align}
 \underset{n\to \infty}{\limsup} \;\; \prob(\|\theta_0\|_2^2 \notin C(\alpha); \cG) \le \alpha\,.
 \end{align}
 Hence,
 \[
  \underset{n\to \infty}{\limsup} \;\; \prob(\|\theta_0\|_2^2 \notin C(\alpha)) \le \alpha +  \underset{n\to \infty}{\limsup} \;\; \prob(\cG^c) \le \alpha \,.
 \]

}

 \subsection{Choice of $U$ for testing $\theta_{\min}$ condition}~\label{app:justification}
 Here we provide a justification for the choice of $U$, given by~\eqref{eq:u-beta-min}, for testing $\theta_{\min}$ condition. Recall that in this case $\Omega_0 = \{\theta\in \reals^p: \min_{j\in \supp(\theta)} |\th_j| \ge c\}$. 
 Instead of directly solving optimization~\eqref{opt:final0}, which is hard due to non-convexity of $\Omega_0$, we first develop a lower bound and find $U$ that maximizes the lower bound.
 
 The lower bound is obtained by fixing $k=1$ in the optimization~\eqref{opt:final0}. The problem then amounts to $$\maximize _{u: \|u\|_2\le 1}\,\, \de(\hth^{(1)},\Omega_0;u)\,,$$ which by plugging in for $\de(\hth^{(1)},\Omega_0;u)$ is equivalent to $$\maximize _{u: \|u\|_2\le 1} \inf_{\theta\in \Omega_0} \,\,|u^\sT(\theta -\hth^{(1)})|\,.$$
 
 We claim that the optimal $u$ should be one of the standard basis element. To see this, consider $u\neq e_i$, for $i\in [p]$. Then, there exists a vector $v\in \reals^p$ such that $v_j\neq 0$ for all $j\in [p]$ and $v^\sT u = 0$. Choose $\lambda\in \reals$ large enough such that all the coordinates of $\th =  \hth^{(1)} + \lambda v$ have magnitude larger than $c$. Therefore, $\th\in \Omega_0$  and $u^\sT(\theta -\hth^{(1)})= 0$.
 
 Setting $u = e_i$, the objective becomes $\inf_{\theta\in \Omega_0} |\theta_i - \hth^{(1)}_i| = |\cS(\hth_i^{(1)},c)-\hth_i^{(1)}|$, by Lemma~\ref{proj-betamin}. Therefore, the optimal value of objective is achieved for $i = i^\star$ given by~\eqref{eq:u-beta-min}.

\subsection{Proof of Lemma~\ref{ssLasso}}\label{app:ssLasso}
We apply~\cite[Theorem 1]{SZ-scaledLasso}, where using their notation with their $\lambda_0$ replaced by $\lambda$, $\xi = 3$, $T =\supp(\th_0)$,
$\kappa(\xi,T)\ge \phi_0$, $\eta_*(\sigma^*\lambda,\xi) \le 4s_0\lambda^2/\phi_0^2$. By a straightforward manipulation of Eq. (13) in~\cite{SZ-scaledLasso}, we have
for $\|X^\sT w/(n\sigma^*)\|_{\infty}\le \lambda/2$, 
\begin{align}
\Big|\frac{\hsigma}{\sigma^*}-1 \Big| \le \frac{2\sqrt{s_0}\lambda}{\phi_0\sigma^*} = \frac{2c}{\phi_0\sigma^*} \sqrt{\frac{\log p}{n}}\,.
\end{align}
Note that
\begin{align}\label{SS0}
\prob\left(\frac{\|X^\sT w\|_{\infty}}{{n}\sigma^*}> \frac{\lambda}{2}\right) \le \prob\left(\frac{\|X^\sT w\|_\infty}{{n}\sigma}> \frac{\lambda}{4}\right) + \prob\left(\frac{\sigma}{\sigma^*}>2\right)
\end{align}
We define $v_j = w^\sT Xe_j/(\sqrt{n}\sigma)$. Since $v_j\sim\normal(0,\hSigma_{jj})$ by applying a standard tail bound on the supremum of $p$ gaussian random variables, we get
\begin{align}\label{SS1}
 \prob\left(\frac{\|X^\sT w\|_\infty}{{n}\sigma}> \frac{\lambda}{4}\right) \le 2pe^{-\lambda^2 n /(32K^2)} = 2p^{-c_0}\,\quad \quad c_0 = \frac{c^2}{32K}-1\,.
\end{align}
For the second term, note that 
$$\frac{\sigma^{*2}}{\sigma^2} = \frac{\|w\|^2}{n\sigma^2} = \frac{1}{n} \sum_{j=1}^n Z_j^2\,,$$
with $Z_j\sim\normal(0,1)$  independent. By a standard tail bound for $\chi^2$ random variables
we have 
\begin{align}\label{SS2}
 \prob\Big(\frac{\sigma^*}{\sigma}\le\frac{1}{2}\Big)\le \prob\Big(\Big|\frac{1}{n} \sum_{j=1}^n Z_j^2-1\Big|>\frac{3}{4}\Big)\le 2e^{-n/16}\,.
\end{align}
Combining~\eqref{SS1}, \eqref{SS2} in \eqref{SS0}, we get that
\[ \prob\left(\frac{\|X^\sT w\|_\infty}{{n}\sigma}> \frac{\lambda}{4}\right) \le 2p^{-c_0} + 2e^{-n/16}\,,\]
which yields the desired result.
%
%
%
\subsection{Proof of Proposition~\ref{pro:mu-random}}\label{proof:mu-random}
Note that by Definition~\ref{mu*}, clearly
\begin{align}
\mu_{\min}(X;U)\le \big|\hSigma\Sigma^{-1} U-U\big|_\infty\, .
\end{align}
Therefore the statement follows readily from the following lemma.
\begin{lemma}\label{lem:main_lem}
Consider a random design matrix $X \in \reals^{n\times p}$, with i.i.d. rows  having mean zero and population covariance $\Sigma$.
Assume that
\begin{itemize}
\item[$(i)$] We have $\sigma_{\min} (\Sigma) \ge C_{\min} >0$, and $\sigma_{\max}(\Sigma) \le C_{\max} <\infty$.
\item[$(ii)$] The rows of $X\Sigma^{-1/2}$ are sub-gaussian with $\kappa = \|\Sigma^{-1/2} x_1\|_{\psi_2}$.
\end{itemize}
Let $\hSigma = (X^\sT X)/n$ be the empirical covariance. Then, for any fixed $U\in \reals^{p\times k}$ independent of $X$ satisfying $U^\sT U = I$, and for any fixed constant $a>0$, the following holds true
\begin{eqnarray}
\prob \bigg\{\Big| \hSigma \Sigma^{-1} U - U \Big|_\infty \ge a \sqrt{\frac{\log p}{n}} \bigg\} \le 2p^{-c_2}\,, 
\end{eqnarray}
with $c_2 = (a^2 C_{\min})/(24 e^2 \kappa^4 C_{\max}) - 2$.
\end{lemma}
\begin{proof}[Proof of Lemma \ref{lem:main_lem}]
The proof is an application of the Bernstein-type inequality for sub-exponential random variables~\cite{Vershynin-CS}.
Define $\tilde{x}_\ell = \Sigma^{-1/2} x_\ell$, for $\ell \in [n]$, and write
\[
H\equiv \hSigma \Sigma^{-1} U - U = \frac{1}{n} \sum_{\ell=1}^n \Big\{x_\ell x_\ell^\sT \Sigma^{-1} U - U \Big\}
= \frac{1}{n} \sum_{\ell=1}^n \Big\{\Sigma^{1/2} \tilde{x}_\ell \tilde{x}_\ell^\sT \Sigma^{-1/2}U - U \Big\}\,.
\] 
Fix $i,j \in [p]$, and for $\ell \in [n]$, let 
$v^{(ij)}_\ell = (e_i^\sT \Sigma^{1/2} \tilde{x}_\ell ) (\tilde{x}_\ell^\sT \Sigma^{-1/2} u_j)  - u_{j,i}$, where $u_{j,i}$ denotes the $i$-th component of $u_j$.
Notice that $\E(v^{(ij)}_\ell) = 0$, and the $v^{(ij)}_\ell$ are independent for $\ell \in [n]$, since $U$ is independent of $X$. In addition, $H_{i,j} = (1/n) \sum_{\ell=1}^n v^{(ij)}_\ell$.
By~\cite[Remark 5.18]{Vershynin-CS}, we have 
\[
\|v^{(ij)}_\ell\|_{\psi_1} \le 2 \|(e_i^\sT \Sigma^{1/2} \tilde{x}_\ell) (\tilde{x}_\ell^\sT \Sigma^{-1/2} u_j) \|_{\psi_1}.
\]
Moreover, for any two random variables $X$ and $Y$, we have
\begin{align*}
\|XY\|_{\psi_1} &= \sup _{p \ge 1} \,  p^{-1} \E(|XY|^p)^{1/p} \\
& \le \sup _{p \ge 1} \, p^{-1} \E(|X|^{2p})^{1/2p}\, \E(|Y|^{2p})^{1/2p} \\
&\le 2\, \Big(\sup _{q \ge 2} \, q^{-1/2} \E(|X|^{q})^{1/q}\Big) \, \Big(\sup _{q \ge 2} \, q^{-1/2} \E(|Y|^{q})^{1/q} \Big)\\
& \le 2 \|X\|_{\psi_2} \, \|Y\|_{\psi_2} \,.
\end{align*}
Hence, by assumption $(ii)$, we obtain
\begin{align*}
\|v^{(ij)}_\ell\|_{\psi_1} &\le 4 \|e_i^\sT \Sigma^{1/2} \tilde{x}_\ell \|_{\psi_2} \|\tilde{x}_\ell^\sT \Sigma^{-1/2} u_j \|_{\psi_2}\\
&\le 2 \|\Sigma^{1/2} e_i\|_2  \|\Sigma^{-1/2} u_j\|_2 \kappa^2 \\
&\le 2 \sqrt{\frac{C_{\max}}{C_{\min}}}\,\|u_j\|_2 \kappa^2 =  2 \sqrt{\frac{C_{\max}}{C_{\min}}}\, \kappa^2\,.
\end{align*}
Define $\kappa'  \equiv 2 \sqrt{C_{\max} /C_{\min}} \kappa^2$. We now use the Bernstein-type inequality for centered sub-exponential random variables~\cite{Vershynin-CS} to get
\[
\prob\Big\{\frac{1}{n} \Big|\sum_{\ell=1}^n v^{(ij)}_\ell \Big| \ge \eps \Big\} \le 2 \exp \Big[ -\frac{n}{6} \min\Big((\frac{\eps}{e\kappa'})^2, \frac{\eps}{e\kappa'}\Big) \Big]\,.
\]
Choosing $\eps = a\sqrt{(\log p)/n}$, and 
assuming $n \ge [a/(e\kappa')]^2 \log p$, we arrive at
\[
\prob\bigg \{\frac{1}{n} \Big|\sum_{\ell=1}^n v^{(ij)}_\ell \Big| \ge a\sqrt{\frac{\log p}{n}}  \bigg\} 
\le 2 p^{-a^2/(6e^2\kappa'^2)}\,.
\]
The result follows by union bounding over all possible pairs $i, j \in [p]$.
\end{proof}

\subsection{Proof of Lemma~\ref{lem:random-deter}}\label{proof:random-deter}
Define the event 
\begin{align}
\cH_n(a) \equiv \Big\{X\in \reals^{n\times p}:\, \Big| \hSigma \Sigma^{-1} U - U \Big|_\infty \le a \sqrt{\frac{\log p}{n}} \Big\}\,.
\end{align}
In other words, $\cH_n(a)$ is the event that $\Sigma^{-1} u_i$ is a feasible solution
of~\eqref{OPT:M-n}, for $1\le i\le k$. By Lemma~\ref{lem:main_lem}, $\prob(\cH_n(a)) \ge 1-2p^{-c_2}$.
On this event, letting $g_i$ be the solution of the optimization
problem~\eqref{OPT:M-n}, we have
\begin{align*}
g_i^\sT \hSigma g_i &\le u_i^\sT\Sigma^{-1} \hSigma \Sigma^{-1} u_i \nonumber \\
& = (u_i^\sT\Sigma^{-1} \hSigma \Sigma^{-1} u_i -
u_i^\sT\Sigma^{-1} u_i)+ u_i^\sT\Sigma^{-1} u_i \nonumber \\
& = \frac{1}{n} \sum_{j=1}^n(V_j^2-
u_i^\sT\Sigma^{-1} u_i)+ u_i^\sT\Sigma^{-1} u_i \,,
\end{align*}
where $V_j = u_i^\sT \Sigma^{-1} x_j$ are i.i.d. random variables with $\E(V_j^2) =
u_i^\sT \Sigma^{-1} u_i$ and sub-gaussian norm 
$$\|V_j\|_{\psi_2} \le \|\Sigma^{-1/2} u_i\|_2\|
\Sigma^{-1/2} x_j\|_{\psi_2}\le \frac{\kappa}{\sqrt{C_{\min}}}.$$

Letting $S_j=V_j^2-
u_i^\sT \Sigma^{-1} u_i $, we have that $S_j$ is zero mean and 
sub-exponential with $\|S_j\|_{\psi_1}\le 2 \|V_j^2\|_{\psi_1} \le 4 \|V_j\|_{\psi_2}^2\le 
4\kappa^2C_{\rm min}^{-1}\equiv \kappa'$. 
Hence, by applying Bernstein inequality for centered sub-exponential random variables~\cite{Vershynin-CS} (similar to the
proof of Lemma \ref{lem:main_lem}), we have, 
for $\eps\le e\kappa'$,
\begin{align*}
\prob\Big(g_i^\sT \hSigma g_i \ge u_i^\sT\Sigma^{-1} u_i+\eps\Big)\le
2\, e^{-(n/6)(\eps/e\kappa')^2}+2\, p^{-c_2}\, .
\end{align*}
We can make $c_2\ge 2$ by a suitable choice of $a$.
The result then follows by letting $\eps = e\kappa'\sqrt{6c_2 (\log p)/n}$.
%
\rev{
\subsection{Proof of Lemma~\ref{lem:moment}}\label{proof:lem-moment}
By definition of sub-gaussian norm, given by~\eqref{def:subG}, we have $\E(|X_{ij}|^q)\le C^qq^{q/2}$, for all $q\ge 1$.
To prove $(i)$, note that $\bE(y_i^2)\le \E(y_i^4)^{1/2} \le \sqrt{C'}$,  and $\E_n[X_{ij}^2 y_i^2]\le (\E_n[X_{ij}^4])^{1/2} (\E_n[y_i^4])^{1/2} \le 4\sqrt{C'} C^2$.
To prove $(ii)$, note that by Holder's inequality we have that for any fixed $j\in [p]$, $\E_n[|X_{ij}^3 w_i^3|]\le (\E_n[X_{ij}^{12}])^{1/4} (\E_n[w_i^4])^{3/4}\le 12^{3/2} C^3 C''^{3/4} = O(1)$ and also by our assumption $\log p  = o(n^{1/3})$. 
Finally to show $(iii)$, we note that by simple union bounds and tail properties of sub-gaussian variables, $\max_{ij} X_{ij}^2 = O_p(\log p)$. Further, $s = o(n/\log^2(p))$ and hence the first part of $(iii)$ holds. 
To prove the second part of $(iii)$, we note that $\max_{j\in [p]} \E_n[X_{ij}^4 w_i^4]\le \E_n[w_i^4] \max_{i\in[n], j\in[p]} X_{ij}^4 = O_P(\log^2 p)$. Now by an application of maximal inequality (See~\cite[Lemma S.4]{belloni2012sparse}), we obtain 
\[
\max_{j\in[p]} \big|\E_n[X_{ij}^2 w_i^2] - \bE[X_{ij}^2 w_i^2] \big| =O_P\Big(\log p \sqrt{\frac{\log p}{n}}\Big) = O_P\Big(\sqrt{\frac{\log^3(p)}{n}}\Big) = o_P(1).
\] 
Likewise, we have 
\[
\max_{j\in[p]} \big|\E_n[X_{ij}^2 y_i^2] - \bE[X_{ij}^2 y_i^2] \big| = o_P(1)\,. 
\]
Combining these two equations we get the second part of $(iii)$.

}

\bibliographystyle{amsalpha}
\bibliography{all-bibliography}

\end{document}